\documentclass{article}

\usepackage{arxiv}

\usepackage{amsmath, amssymb, amscd, amsthm, amsfonts}
\usepackage{graphicx}
\usepackage{hyperref}
\usepackage{cleveref}
\usepackage{proba}
\usepackage{bm}
\usepackage{xcolor}
\usepackage[authoryear]{natbib}
\usepackage[title]{appendix}
\usepackage{afterpage}
\usepackage{subcaption}
\usepackage{enumitem}

\usepackage{subfiles} 

\usepackage{setspace}
\title{Mixture of experts models for multilevel data: modelling framework and approximation theory}
\author{ 
	Tsz Chai Fung\thanks{Department of Risk Management and Insurance, Georgia State University. 35 Broad Street NW, Atlanta, GA 30303, United States. Email address: \texttt{tfung@gsu.edu}.} \\
	\And
    Spark C.~Tseung\thanks{Department of Statistical Sciences, University of Toronto. Ontario Power Building, 700 University Avenue, 9th Floor, Toronto, ON M5G 1Z5, Canada. Email addresses: \texttt{spark.tseung@mail.utoronto.ca}. 
    }\\ 
}
\date{ }

\hypersetup{
pdftitle={Mixture of experts models for multilevel data: modelling framework and approximation theory},
pdfsubject={},
pdfauthor={},
pdfkeywords={Artificial neural network, Crossed and nested random effects, Denseness; Mixed effects models, Universal approximation theorem},
}

\newtheorem{assumption}{Assumption}
\newtheorem{theorem}{Theorem}
\newtheorem{lemma}{Lemma}
\newtheorem{remark}{Remark}

\newtheorem{definition}{Definition}

\begin{document}

\maketitle

\begin{abstract}
Multilevel data are prevalent in many real-world applications. However, it remains an open research problem to identify and justify a class of models that flexibly capture a wide range of multilevel data. Motivated by the versatility of the mixture of experts (MoE) models in fitting regression data, in this article we extend upon the MoE and study a class of mixed MoE (MMoE) models for multilevel data. Under some regularity conditions, we prove that the MMoE is dense in the space of any continuous mixed effects models in the sense of weak convergence. As a result, the MMoE has a potential to accurately resemble almost all characteristics inherited in multilevel data, including the marginal distributions, dependence structures, regression links, random intercepts and random slopes. In a particular case where the multilevel data is hierarchical, we further show that a nested version of the MMoE universally approximates a broad range of dependence structures of the random effects among different factor levels.
\end{abstract}



\section{Introduction} \label{sec:intro}

Mixture of experts (MoE) model, which is first introduced by \cite{jacobs1991adaptive} (see also, e.g., \cite{jordan1994hierarchical} and \cite{mclachlan2000finite} for details), is a probabilistic version of neural network architecture useful for flexible regression, classification and distribution modelling, with applications to various areas including healthcare, business, social and environmental science. Readers may refer to \cite{yuksel2012twenty}, \cite{masoudnia2014mixture} and \cite{nguyen2018practical} for the literature reviews on both the theories and applications of the MoE.

The model structure of the MoE is as follows. Suppose that we have $N$ observations $(\bm{y},\bm{x})=\{(\bm{y}_i,\bm{x}_i)\}_{i=1,\ldots,N}$, where $\bm{y}_i=(y_{i1},\ldots,y_{iK})$ is a $K$-dimensional response variable with output space $\mathcal{Y}\subseteq\mathbb{R}^K$ and $\bm{x}_i=(x_{i1}\ldots,x_{iP})$ are the $P$ covariates or features with input space $\mathcal{X}\subseteq\mathbb{R}^P$. Under the MoE framework, the conditional distribution function of $\bm{y}_i$ given $\bm{x}_i$ is

\begin{equation} \label{eq:moe}
F(\bm{y}_i;\bm{\alpha},\bm{\psi},g|\bm{x}_i)=\sum_{j=1}^{g}\pi_j(\bm{x}_i;\bm{\alpha})F_0(\bm{y}_i;\bm{\psi}_j|\bm{x}_i),
\end{equation}
where $g$ is the number of latent classes. Here, $\pi_j(\bm{x}_i;\bm{\alpha})>0$ is called the gating function with $\sum_{j=1}^{g}\pi_j(\bm{x}_i;\bm{\alpha})=1$ and parameters $\bm{\alpha}$. While the most common choice of gating function is the logit-linear or softmax gating (\cite{jacobs1991adaptive}) given by $\pi_j(\bm{x}_i;\bm{\alpha})=\exp\{\alpha_{j,0}+\bm{\alpha}_j^T\bm{x}_i\}/\sum_{j'=1}^{g}\exp\{\alpha_{j',0}+\bm{\alpha}_{j'}^T\bm{x}_i\}$ with $\bm{\alpha}=\{\alpha_{j0},\bm{\alpha}_{j}:j=1,\ldots,g\}$, alternative gating functions such as Gaussian gating (\cite{xu1995alternative}), student-t gating (\cite{ingrassia2012local}) and probit gating (\cite{geweke2007smoothly}) are also explored in literature. Also, $F_0(\bm{y}_i;\bm{\psi}_j|\bm{x}_i)$ is a probability distribution called the expert function with parameters $\bm{\psi}:=\{\bm{\psi}_j:j=1,\ldots,g\}$. While a common choice for the expert function is a Gaussian distribution (\cite{jordan1992hierarchies}), there has been substantial developments on alternative choices of expert functions to cater for various distributional characteristics such as heavy-tailedness (Laplace by \cite{nguyen2016laplace}, t-distribution by \cite{chamroukhi2016robust}, skewed t by \cite{chamroukhi2017skew} and transformed gamma by \cite{fung2020moe3}) and discrete distributions (Poisson by \cite{grun2008flexmix} and Erlang Count by \cite{fung2019moe2}).

Model flexibility is a crucial desirable property for the class of MoE, and there are extensive research work on the approximation theory for the MoE. \cite{zeevi1998error} shows that the mean function of a univariate ($K=1$) logit-gated MoE can approximate any Sobolev class functions. This result is extended by \cite{jiang1999approximation}, who considers the transformed Sobolev class. Without considering the convergence rate, \cite{nguyen2016universal} proves that the MoE mean function is dense in the class of any continuous functions without the restriction of the Sobolev class, and \cite{nguyen2019approximation} shows similar denseness results using a multivariate ($K>1$) Gaussian-gated MoE.

Apart from studying the mean functions, some research works focus on conditional density approximation with respect to the Hellinger distance, Kullback–Leibler (KL) divergence, or Lebesgue space. \cite{jiang1999hierarchical} and \cite{mendes2012convergence} generalize the results of \cite{jiang1999approximation} by demonstrating the approximation capability of the MoE to any exponential family non-linear regression models. \cite{norets2010approximation} shows that the logit-gated MoE with Gaussian expert function can approximate any conditional densities. Similar results are proved by \cite{norets2014posterior} and \cite{nguyen2019approximation}, who consider the Gaussian gating functions. Recently, \cite{nguyen2021approximations} proves that the class of MoE is dense in the Lebesgue space.

Another stream of distribution approximation theorems studies the denseness in the sense of Prohorov metric of weak convergence. Extending upon \cite{tijms1994stochastic} and \cite{breuer2005introduction} who explore denseness of finite mixture and phase-type distributions in the space of any probability distributions, \cite{fung2019moe1} formulates the concept of ``denseness" in regression settings and shows that the class of MoE is dense in the space of any regression distributions, subject to some regularity conditions such as Lipschitz continuity and distribution tightness. In contrast to other existing works on distribution approximations, the results of \cite{fung2019moe1} are very general as: (i) they hold under a wide range of choices of expert functions (not restricted to Gaussian or other symmetric expert functions); (ii) the target distribution is not restricted to a special class (e.g., exponential family regression models).

Despite of its model flexibility, the aforementioned modelling framework implicitly assumes that the input-output pairs are independent among observations. This is not true for multilevel data (\cite{goldstein2011multilevel}). Apart from the inputs $\bm{x}_i$, there are also $L$ levels of factors $\bm{\theta}_i=(\bm{\theta}_{i1},\ldots,\bm{\theta}_{iL})$ jointly affecting the output $\bm{y}_i$. While these factors are unobserved, they are clustered into various units for each level $l=1,\ldots,L$ and we know how they are clustered. As shown in Figure \ref{fig:multilevel}, for each level $l=1,\ldots,L$, each observation $i$ is assigned into one of the $S_l$ units, where $c_l(\cdot):\{1,\ldots,N\}\mapsto\{1,\ldots,S_l\}$ is denoted as a known function mapping an observation index to one of the $S_l$ units, with $\bm{c}(i)=(c_1(i)\ldots,c_L(i))$. We have $\bm{\theta}_{il}=\bm{\theta}_{i'l}:=\bm{\theta}_l^{(s)}$ if $c_l(i)=c_l(i')=s$. As the factor $\bm{\theta}_l^{(s)}$ affects multiple observations at the same time, there are interdependencies among observations. Ignoring such a dependency would lead to spurious, misleading or biased clustering and prediction outcomes (\cite{goldstein2011multilevel}).

\afterpage{
\begin{figure}[!h]
\begin{center}
\includegraphics[width=\linewidth]{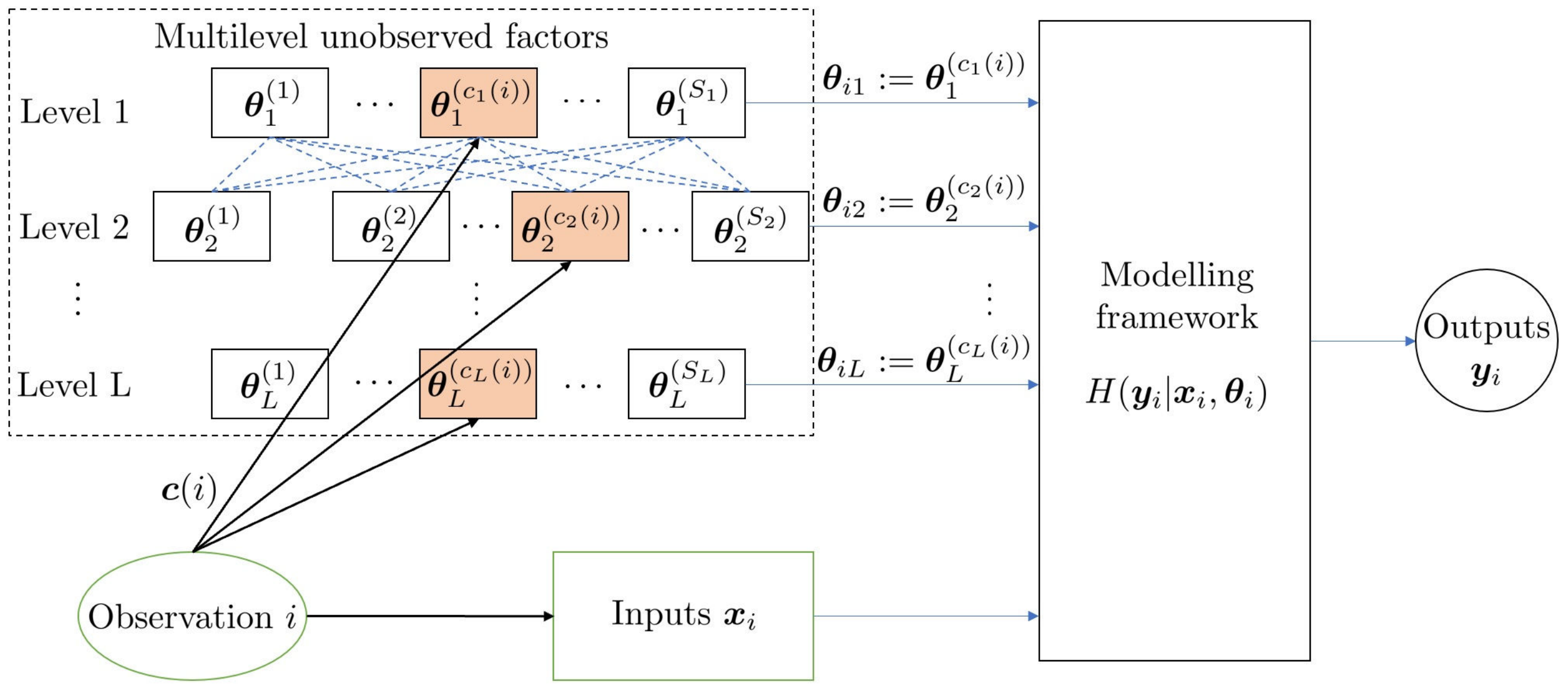}
\end{center}
\caption{Multilevel data structure and modelling framework.}
\label{fig:multilevel}
\end{figure}
}

Multilevel data are prevalent across many applications. The most classic one is the school problem (\cite{aitkin1986statistical}, \cite{goldstein1986multilevel} and \cite{frees2006multilevel}), where ``school" and ``classroom within a school" act as two levels of factors affecting the performance of a student (as an observation). Multilevel data structure can also be caused by repeated measurements collected in longitudinal studies. This is common across various areas including health (e.g., \cite{molenberghs2010family}) and business (e.g., \cite{boucher2006fixed}). For instance, the medical outcome of a patient or the amount of insurance claims by a policyholder are measured or collected repeatedly over time. A remarkable special case of multilevel data is the hierarchical (or nested) data, where the $L$ factor levels are ranked from high to low, and every lower level factor belongs to a particular upper level factor. The school problem is a clear example of hierarchical data.

A popular model to account for the interdependencies among observations is the generalized linear mixed effects model (GLMM) (\cite{goldstein1986multilevel} and \cite{mcgilchrist1994estimation}), which assumes that the output $\bm{y}_i$ depends on the sum of fixed effects (i.e., the impact of the inputs $\bm{x}_i)$ and random effects (i.e., the impact of the factors $\bm{\theta}_i$). To improve model flexibility or achieve specific clustering purposes, the GLMM framework is extended to a non-linear setting (\cite{davidian1993nonlinear} and \cite{gregoire1996non}), formulated in a neural network structure (\cite{bakker2003task}) or integrated to a finite-mixture modelling framework (\cite{ng2004modelling} and \cite{ng2006mixture}). Despite of the desirable properties of the MoE models leading to extensive applications, the research works of mixed effects models in the context of MoE framework are relatively scarce. \cite{yau2003finite} first proposes a two-component logit-gated Gaussian-expert MoE with random effects incorporated in both gating and expert functions. \cite{ng2007extension} then formulates a general $g$-component mixed effect MoE with the use of logistic expert functions for binary classifications. \cite{ng2014mixture} considers a similar framework with random effects only incorporated to the expert functions. Nonetheless, all the aforementioned mixed effect MoE models only deal with a single level of random effect (i.e., $L=1$).

Motivated by the increasing popularity of the MoE models, the prevalence of multilevel data and the desirability of formally justifying model flexibility through approximation theories, the contributions of this paper are fourfold: Firstly, we propose the mixed MoE (MMoE) for multilevel regression data that allows for multiple levels of random effects. Compared to the existing literature (\cite{yau2003finite}, \cite{ng2007extension} and \cite{ng2014mixture}), our proposed model is reduced in two ways: (i) random effects are incorporated only into the gating functions and assumed to be independent, and the model includes the random effects differently; (ii) regression links are removed from the expert functions. Secondly, we formulate the definition of denseness for the class of mixed effects models in the sense of weak convergence, which directly extends upon \cite{fung2019moe1} who formulate denseness for regression distributions. Thirdly, we prove that the class of MMoE is dense in the space of any continuous mixed effects models subject to some mild regularity conditions. This not only demonstrates the flexibility of the proposed model in capturing various aspects of multilevel data characteristics such as joint distributions, regression patterns, random intercepts, and random slopes, but also suggests that our proposed model is parsimonious with a reduced structure. Compared to \cite{fung2019moe1}, several assumptions for the denseness theorem are also relaxed in this paper. For example, Lipschitz continuity and distribution tightness are no longer explicitly required. Hence, the proof techniques of this paper are quite different from those in \cite{fung2019moe1}. Fourthly, in a particular case of hierarchical data, we further prove that a nested version of the MMoE can accurately approximate a wide range of dependence structures between the upper and lower level factors, even if the MMoE considered is a simplified model class consisting of only independent random effects across levels. This result is critical for many applications, e.g., the classroom impact on a student's performance is likely dependent on the school the student attends.

The focus of this paper is to formulate the MMoE model for multilevel data and theoretically justify its versatility. In a subsequent paper (\cite{Tseung2022MixedLRMoEApplication}), we will address the estimation and application problems under the proposed MMoE. A stochastic variational ECM algorithm is proposed to efficiently estimate the model parameters. Also, the MMoE is applied to an automobile insurance dataset, demonstrating its ability to reasonably predict policyholders' future claims based on their past claim histories.

This paper is structured as follows. Section \ref{sec:gen} defines a generalized class of mixed effect models for multilevel data, which includes nearly all mixed effect models in the literature. In Section \ref{sec:moe}, we introduce the MMoE as a candidate class of mixed effect models to flexibly capture multilevel data. Interpretation and visualization of the proposed model are also provided. Section \ref{sec:dense} defines ``denseness" in the context of mixed effect models and proves that the MMoE is a universal approximator of most mixed effect models subject to some mild conditions. In Section \ref{sec:nested}, we discuss the model formulation and denseness property in a special case where the dataset is hierarchical with nested random effects. The findings are summarized in Section \ref{sec:discussions}, accompanying some limitations of the denseness theory in justifying the approximation capability of the proposed MMoE.

\section{Mixed effect models for multilevel data} \label{sec:gen}

Datasets with multilevel structure are often modelled by a mixed effects model. Under this modelling framework, the effects of the known inputs $\bm{x}_i$ on the output $\bm{y}_i$ are perceived as the ``fixed effect" or ``hard sharing of parameters", while the $L$ levels of unobserved factors $\bm{\theta}_i$ are treated as random (specified by a distribution) and their impacts on $\bm{y}_i$ are regarded as ``random effect" or ``soft sharing of parameters". In this section, we will discuss some technical details regarding a generalized framework of the mixed effects models.

Let $(\bm{\Omega},\mathcal{F},\mathbb{P})$ be the probability space and suppose that $\bm{\theta}_l^{(s)}$ is a $\mathcal{F}$-measurable map from $(\bm{\Omega},\mathcal{F})$ to $(\bm{\Theta}_l,\mathcal{Q}_l)$ for every $l=1,\ldots,L$ and $s=1,\ldots,S_l$, where $\bm{\Theta}_l$ is the space of $\bm{\theta}_l^{(s)}$ or $\bm{\theta}_{il}$. $\bm{\theta}_l^{(s)}$ is a random variable if $(\bm{\Theta}_l,\mathcal{Q}_l)=(\mathbb{R},\mathcal{R})$ where $\mathcal{R}$ is a Borel set, but we do not want to impose such a restriction. It is because the factors $\bm{\theta}_l^{(s)}$ are unobserved and we are unsure if these factors can be quantified as a real number. Using clinical studies as an example, it is impossible to summarize the unobserved characteristics of hospitals or doctors into just a single number as their impacts to a patient's medical outcome are complicated, possibly related to many unknown hidden factors including hospital's medical equipment and financial support, and doctor's education and expertise. The space $\bm{\Theta}_l$ also varies among different mixed effects models in literature. For example, $\bm{\Theta}_l=\mathbb{R}$ for most GLMMs (\cite{mcgilchrist1994estimation}), $\bm{\Theta}_l=\mathbb{R}^{2g-1}$ for the GLMM MoE model by \cite{ng2007extension}, and $\bm{\Theta}_l=\mathbb{R}^{gK}$ for the mixture of random effects models by \cite{ng2014mixture}.

Under the generalized mixed effects model, we assume that $\bm{y}_i$ is influenced directly only by $\bm{x}_i$ and $\bm{\theta}_i$. Conditioned on $\bm{x}_i$ and $\bm{\theta}_i$, it is further assumed that $\{\bm{y}_i\}_{i=1,\ldots,N}$ are mutually independent. In particular, we have
\begin{equation} \label{eq:gen_y_conditional}
    \bm{y}_i|\bm{x}_i,\bm{\theta}_i\overset{\text{ind}}{\sim} H(\cdot|\bm{x}_i,\bm{\theta}_i),\qquad i=1,\ldots,N,
\end{equation}
where $H$ can be any probability distributions. Figure \ref{fig:multilevel} shows a visualization of the modelling framework. With these assumptions, the joint distribution of $\bm{y}$ given $\bm{x}$ is given by
\begin{equation} \label{eq:gen_joint_cross}
\tilde{H}(\bm{y}|\bm{x})=\int_{\Omega}\left[\prod_{i=1}^{N}H(\bm{y}_i|\bm{x}_i,\bm{\theta}_i)\right]d\mathbb{P}
\end{equation}

Suppose that each measurable space $(\bm{\Theta}_l,\mathcal{Q}_l)$ is also equipped by a probability measure $G_l$, corresponding to the ``distribution" of $\bm{\theta}_l^{(s)}$. Denote further $(\tilde{\bm{\Theta}},\tilde{\mathcal{Q}},G)$ as the product of probability spaces $\{(\bm{\Theta}_l,\mathcal{Q}_l,G_l)\}_{l=1,\ldots,L;s=1,\ldots,S_l}$. Then, the joint distribution can be re-written as
\begin{equation} \label{eq:gen_joint_cross2}
\tilde{H}(\bm{y}|\bm{x})=\int_{\tilde{\bm{\Theta}}}\left[\prod_{i=1}^{N}H(\bm{y}_i|\bm{x}_i,\bm{\theta}_i)\right]dG(\tilde{\bm{\theta}}),
\end{equation}
where $\tilde{\bm{\theta}}=\{\bm{\theta}_l^{(s)}\}_{l=1,\ldots,L;s=1,\ldots,S_l}$. Note that the above model framework includes a very wide range of models for multilevel or hierarchical data, including generalized linear mixed models (GLMM) and non-linear mixed effects models (see, e.g., \cite{goldstein1986multilevel} and \cite{davidian1993nonlinear}). As there are no restrictions on the functional forms of $H$ and $G$, the above model structure is indeed very general, automatically containing any possible joint distributions of $\bm{y}_i|\bm{x}_i,\bm{\theta}_i$, regression links between $\bm{x}_i$ and $\bm{y}_i$ (including non-linear effect and interactions among covariates), effects of unobserved factors $\bm{\theta}_i$ alone on $\bm{y}_i$ (called random intercept), and interactions between $\bm{\theta}_i$ and $\bm{x}_i$ (called random slope).

Since we have defined $(\tilde{\bm{\Theta}},\tilde{\mathcal{Q}},G)$ as a product probability space, the following assumption has been implicitly made on $\bm{\theta}_l^{(s)}$:

\begin{assumption} \label{asm:indep}
$\{\bm{\theta}_{l}^{(s)}\}_{l=1,\ldots,L;s=1,\ldots,S_l}$ are mutually independent.
\end{assumption}

Therefore, $G(\tilde{\bm{\theta}})$ can be written as

\begin{equation} \label{eq:gen_random_effect}
G(\tilde{\bm{\theta}})=\prod_{l=1}^{L}\prod_{s=1}^{S_l}G_l(\bm{\theta}_l^{(s)})
\qquad
\text{or}
\qquad
dG(\tilde{\bm{\theta}})=\prod_{l=1}^{L}\prod_{s=1}^{S_l}G_l(d\bm{\theta}_l^{(s)}).
\end{equation}

Note that the prior independence assumption across factors within a level (i.e., $s=1,\ldots,S_l$) is very natural for most data structures especially for those involving repeated measurements (see, e.g., \cite{goldstein1986multilevel}, \cite{yau2003finite}, \cite{ng2004modelling}, \cite{boucher2006fixed}, and \cite{ng2007extension}). The prior independence across levels (i.e., $l=1,\ldots,L$) is also often assumed for datasets with multilevel structure (see, e.g., \cite{goldstein1986multilevel} and \cite{mcgilchrist1994estimation}). 

\section{Mixture of experts model with random effect} \label{sec:moe}

Despite of the generality of the above mixed effect model (Equation (\ref{eq:gen_joint_cross})), it is essential to appropriately specify the functional forms of $H$ and $G$ to model a multilevel dataset. However, this is challenging, especially when the space $\tilde{\bm{\Theta}}$ of the latent factors $\tilde{\bm{\theta}}$ is not observed from the dataset. Recall that a multilevel dataset only provides information on how each observation is classified into one of the factors for each level $l$, but not on what the factors are or how to quantify these factors. 

\afterpage{
\begin{figure}[!h]
\begin{center}
\includegraphics[width=\linewidth]{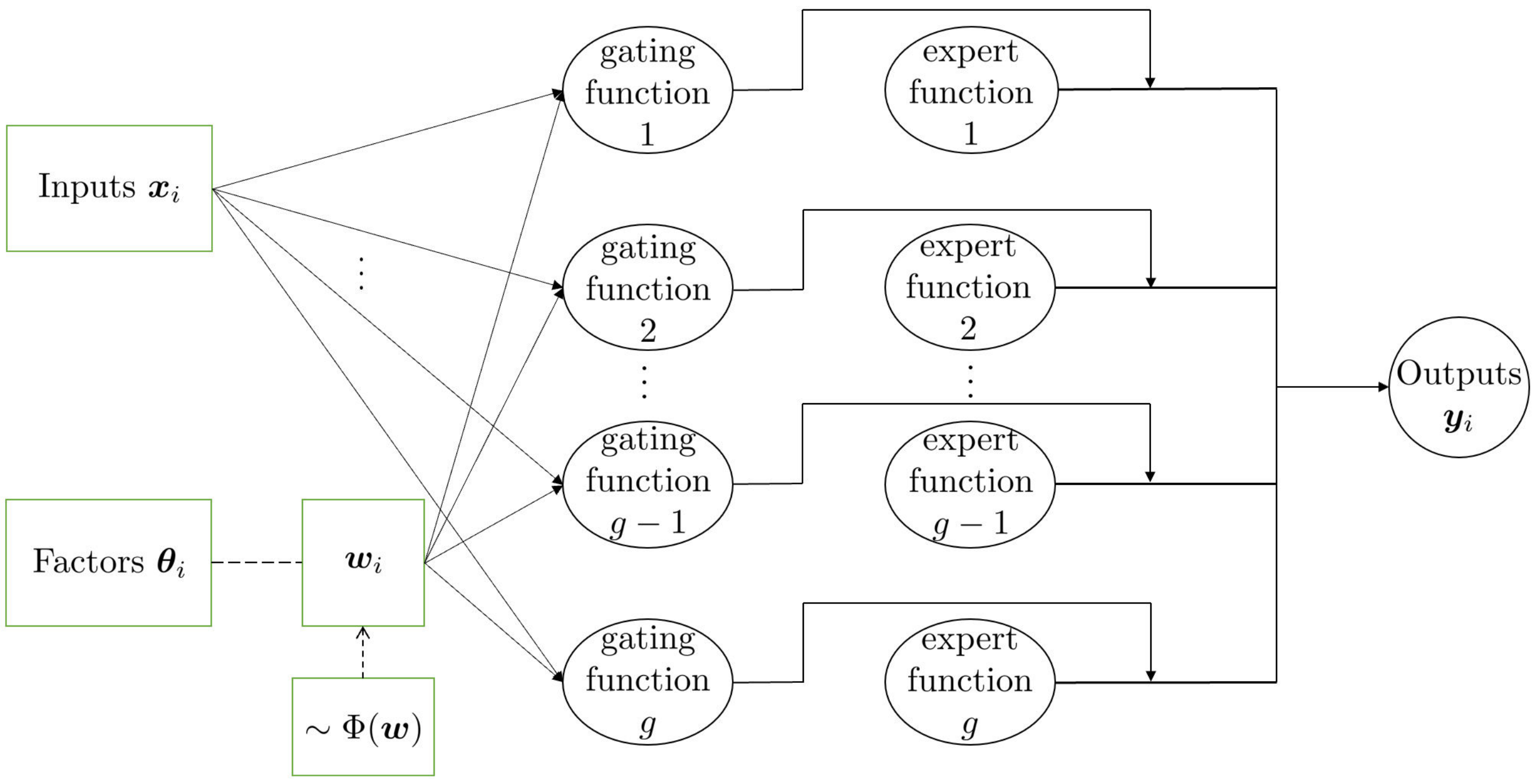}
\end{center}
\caption{Model structure of the MMoE.}
\label{fig:moe}
\end{figure}
}


In this section, we introduce the mixture of experts (MoE) model with random effects, called the mixed MoE (MMoE), as a candidate regression model to cater for multilevel data structure. The justification of the proposed model, which analyzes the ability of the MMoE to accurately approximate the generalized form of the mixed effect model (Equation (\ref{eq:gen_joint_cross})), will be presented in the next section.


Under the MMoE, we assume that each observation $i$ is equipped by $L$ levels of random effects, denoted by an $L$-vector $\bm{w}_i=(w_{i1},\ldots,w_{iL})$. Similar to the mapping of the unobserved factors $\bm{\theta}_i$ introduced in Section \ref{sec:intro}, we also have $w_{il}=w_{i'l}:=w_{l}^{(s)}$ if $c_l(i)=c_l(i')=s$. The only difference between $\bm{\theta}_i$ and $\bm{w}_i$ is that we restrict $w_{il}\in\mathbb{R}$ into a Euclidean space instead of a general space $\tilde{\bm{\Theta}}$, which is unknown and hard to specify. Similar to $\tilde{\bm{\theta}}$, we also define $\bm{w}=\{w_l^{(s)}\}_{l=1,\ldots,L;s=1,\ldots,S_l}$ as the random effects across all levels and factors.


The distribution function of $\bm{y}_i$ conditional on $\bm{x}_i$ and $\bm{w}_i$ is given by
\begin{equation} \label{eq:moe_disn_cross}
F(\bm{y}_i; \bm{\alpha}, \bm{\beta}, \bm{\psi}, g| \bm{x}_{i}, \bm{w}_i) = \sum_{j=1}^{g} \pi_j(\bm{x_i}, \bm{w}_{i}; \bm{\alpha}, \bm{\beta}) F_0(\bm{y}_i; \bm{\psi}_j),
\end{equation}
where $g$ is the number of latent classes, $\pi_j(\bm{x_i}, \bm{w}_{i}; \bm{\alpha}, \bm{\beta})$ is the mixing weight for the $j^{th}$ class (called the gating function), $\bm{\alpha}=\{\alpha_{j0},\bm{\alpha}_{j}:j=1,\ldots,g\}\in\mathcal{A}$ are the regression parameters of the gating function, $\bm{\beta}=\{\bm{\beta}_{j}:j=1,\ldots,g\}\in\mathcal{B}$ are the coefficients of the random effects and $\bm{\psi}=\{\bm{\psi}_j:j=1,\ldots,g\}\in\bm{\Psi}$ are the parameters of a pre-specified multivariate distribution $F_0$ (called the expert function). We also specify $\pi_j(\bm{x}_i,\bm{w}_i; \bm{\alpha}, \bm{\beta})$ as a logit linear gating function, given by

\begin{equation} \label{eq:moe_gate_cross}
    \pi_j(\bm{x}_i,\bm{w}_i; \bm{\alpha}, \bm{\beta}) = \frac{\exp\{\alpha_{j0} + \bm{\alpha}_j^T \bm{x}_i +  \bm{\beta}_j^T \bm{w}_i\}}{\sum_{j'=1}^{g}\exp\{\alpha_{j'0} + \bm{\alpha}_{j'}^T \bm{x}_i +  \bm{\beta}_{j'}^T \bm{w}_i\}}, \quad j = 1, 2, \dots, g.
\end{equation}

Apart from that, the random effects $\{w_{l}^{(s)}\}_{l=1,\ldots,L;s=1,\ldots,S_l}$ are assumed to be independent across $l$ and $s$, and $w_{l}^{(s)}$ follows a fixed pre-specified distribution $\Phi_l$ with no extra parameters in it. Based on the above model specification, the joint distribution of $\bm{y}$ given $\bm{x}$ is

\begin{equation} \label{eq:moe_joint_cross}
\tilde{F}(\bm{y}; \bm{x}):=\tilde{F}(\bm{y}; \bm{\alpha}, \bm{\beta}, \bm{\Psi}, g| \bm{x})=\int\prod_{i=1}^{N} F(\bm{y}_{i}; \bm{\alpha}, \bm{\beta}, \bm{\Psi}| \bm{x}_{i}, \bm{w}_i)d\Phi(\bm{w}),
\end{equation}
where $\Phi$ is the joint distribution of $\bm{w}$, given by

\begin{equation} \label{eq:moe_phi_cross}
\Phi(\bm{w})=\prod_{l=1}^{L}\prod_{s=1}^{S_l}\Phi_l(w_l^{(s)})
\qquad
\text{or}
\qquad
d\Phi(\bm{w})=\prod_{l=1}^{L}\prod_{s=1}^{S_l}\Phi_l(d w_l^{(s)}).
\end{equation}

The model can be interpreted as follows with an illustration displayed in Figure \ref{fig:moe}. Each observation is assigned into one of the $g$ homogeneous subgroups with classification probabilities equal to the gating functions $\pi_j(\bm{x}_i,\bm{w}_i; \bm{\alpha}, \bm{\beta})$. The classification probabilities vary among observations as they depend on both the inputs $\bm{x}_i$ and the unobserved factors $\bm{w}_i$. Conditioned on the subgroup observation $i$ belongs to, the outputs $\bm{y}_i$ are governed by a homogeneous probability distribution $F_0(\bm{y}_i;\bm{\psi}_j)$ independent of $\bm{x}_i$ and $\bm{w}_i$.


One of the noticeable difference between the above model and the standard structure of MoE (i.e., Equation (\ref{eq:moe})) is that here we remove the regression relationship on the expert functions (i.e., the expert functions do not depend on the inputs $\bm{x}_i$), which is considered as a reduced MoE (RMoE) by \cite{fung2019moe1}. Also, note that our model assumes that the level-$l$ random effect $w_{il}$ is the same across all $g$ gating functions (i.e., $w_{il}$ does not depend on $j$). This assumption is different from \cite{ng2007extension}, who considers multiple different independent level-$l$ random effects across the gating functions.

\begin{remark}
The proposed reduced modelling framework may be advantageous in two ways. Firstly, we may select among a wide range of probability distributions as the expert function $F_0$, including more complex non-exponential classes of distributions (e.g., phase-type distributions) where regression modelling on these distributions may be either infeasible or computationally challenging. Secondly, simplified model structure is favourable for interpretation, as our proposed model enables clustering of observations into homogeneous subgroups, and explains the variability of each level-$l$ factor by a single source (i.e., $w_{il}\in\mathbb{R}$) instead of multiple sources by \cite{ng2007extension}. 

\end{remark}

The remaining issue is: how does such a reduced structure affect its model flexibility? To justify the proposed modelling structure, we will demonstrate the denseness property of our proposed model in the following section, which means that the MMoE model structure of Equation (\ref{eq:moe_joint_cross}) can approximate any generalized form of mixed effect models expressed by Equation (\ref{eq:gen_joint_cross}). This will provide evidences suggesting that our proposed model is parsimonious. In other words, the MMoE has the simplest structure without harming its representation capability.


\section{Denseness theory} \label{sec:dense}

This section studies the approximation capability of the class of MMoE models. Our goal is to show that the proposed MMoE is versatile enough to approximate any mixed effects models under mild regularity conditions, even if the MMoE is constructed in a reduced form: (i) the gating function is restricted to be a logit linear gating; (ii) regression link is removed in the expert functions; (iii) the random effects are restricted to follow some fixed pre-determined distributions. Before that, we need to technically formulate a class of mixed effects models and define ``denseness" for mixed effects models. These definitions are the extensions of \cite{fung2019moe1}, who defines ``regression distributions" and ``denseness" in the regression settings without considering random effects. 


Denote $\mathcal{T}:=\mathcal{T}_1\times\ldots\times\mathcal{T}_L$ as a collection of some spaces of $\bm{\Theta}:=\bm{\Theta}_1\times\ldots\times\bm{\Theta}_L$, $\mathcal{H}$ as a collection of some distribution functions $H$ on $(\bm{y}_i|\bm{x}_i,\bm{\theta}_i)$, and $\mathcal{G}_l$ as a collection of probability measures $G_l$ on $\bm{\theta}_l^{(s)}$ with $\mathcal{G}:=\mathcal{G}_1\times\ldots\times\mathcal{G}_L$. Also, let $\mathcal{C}$ be a set containing all possible mappings $\bm{c}(\cdot)$, and define a vector $\bm{S}=(S_1,\ldots,S_L)$ with $\mathcal{S}=\mathbb{N}^L$. ``A class of mixed effects models" and ``mixed effects distributions" are first defined as follows:

\begin{definition} \label{def:model_class}
A class of mixed effects models $\mathcal{M}_L(\mathcal{X};\mathcal{T},\mathcal{H},\mathcal{G}):=\{\tilde{H}(\cdot;\mathcal{X};\bm{\Theta},H,G):\bm{\Theta}_l\in\mathcal{T}_l,H\in\mathcal{H},G_l\in\mathcal{G}_l,l=1,\ldots,L\}$ is a collection of mixed effects distributions $\tilde{H}(\cdot;\mathcal{X};\bm{\Theta},H,G)$, where each mixed effects distribution $\tilde{H}(\cdot;\mathcal{X};\bm{\Theta},H,G):=\{\tilde{H}(\bm{y}|\bm{x}):=\tilde{H}(\bm{y}|\bm{x};\bm{\Theta},H,G)=\int_{\tilde{\bm{\Theta}}}\left[\prod_{i=1}^{N}H(\bm{y}_i|\bm{x}_i,\bm{\theta}_i)\right]dG(\tilde{\bm{\theta}}):\bm{x}_i\in\mathcal{X},i\in\{1,\ldots,N\},N\in\mathbb{N},\bm{S}\in\mathcal{S},\bm{c}\in\mathcal{C}\}$ is itself a collection of joint probability distributions.
\end{definition}

In the spirit of \cite{fung2019moe1}, denseness is defined in the sense of weak convergence of probability distributions. Therefore, before defining denseness, we need to define weak convergence of mixed effects distributions as follows:
\begin{definition} \label{def:convergence}
Consider a sequence of mixed effects distributions $\tilde{\bm{H}}^{(n)}:=\tilde{H}(\cdot;\mathcal{X};\bm{\Theta}^{(n)},H^{(n)},G^{(n)})$ and a target mixed effects distribution $\tilde{\bm{H}}:=\tilde{H}(\cdot;\mathcal{X};\bm{\Theta},H,G)$. We say that $\{\tilde{\bm{H}}^{(n)}\}_{n=1,2,\ldots}$ weakly converge to $\tilde{\bm{H}}$ if and only if for every given $\bm{x}_i\in\mathcal{X}$ (for all $i\in\{1,\ldots,N\}$), $N\in\mathbb{N}$, $\bm{S}\in\mathcal{S}$ and $\bm{c}\in\mathcal{C}$, we have $\tilde{H}(\cdot|\bm{x};\bm{\Theta}^{(n)},H^{(n)},G^{(n)})\xrightarrow{\mathcal{D}}\tilde{H}(\cdot|\bm{x};\bm{\Theta},H,G)$ as $n\rightarrow\infty$, where $\xrightarrow{\mathcal{D}}$ represents a weak convergence or convergence in distribution. If the distributional convergence is uniform across any compact input space $\bar{\mathcal{X}}\subseteq\mathcal{X}$, i.e. $\tilde{H}(\bm{y}|\bm{x};\bm{\Theta}^{(n)},H^{(n)},G^{(n)})\rightarrow\tilde{H}(\bm{y}|\bm{x};\bm{\Theta},H,G)$ uniformly on $(\bm{y},\bm{x})$ with $\bm{x}_i\in\bar{\mathcal{X}}$ (for all $i\in\{1,\ldots,N\}$), then we say that $\{\tilde{\bm{H}}^{(n)}\}_{n=1,2,\ldots}$ weakly converge to $\tilde{\bm{H}}$ compactly.
\end{definition}

Now, we are able to extend the formalism of \cite{fung2019moe1} and define denseness in the settings of mixed effects models:

\begin{definition} \label{def:dense}
Consider two classes of mixed effects models $\mathcal{M}_{1L}:=\mathcal{M}_L(\mathcal{X};\mathcal{T}_1,\mathcal{H}_1,\mathcal{G}_1)$ and $\mathcal{M}_{2L}:=\mathcal{M}_L(\mathcal{X};\mathcal{T}_2,\mathcal{H}_2,\mathcal{G}_2)$. $\mathcal{M}_{1L}$ is dense in $\mathcal{M}_{2L}$ if and only if for every $(\bm{\Theta}_2,H_2,G_2)\in\mathcal{T}_2\times\mathcal{H}_2\times\mathcal{G}_2$, there exists a sequence of $\{(\bm{\Theta}_1^{(n)},H_1^{(n)},G_1^{(n)})\}_{n=1,2,\ldots}$ with $(\bm{\Theta}_1^{(n)},H_1^{(n)},G_1^{(n)})\in\mathcal{T}_1\times\mathcal{H}_1\times\mathcal{G}_1$ such that the mixed effects distributions $\{\tilde{\bm{H}}^{(n)}_1:=\tilde{H}(\cdot;\mathcal{X};\bm{\Theta}_1^{(n)},H_1^{(n)},G_1^{(n)})\}_{n=1,2,\ldots}$ weakly converge to $\tilde{\bm{H}}_2:=\tilde{H}(\cdot;\mathcal{X};\bm{\Theta}_2,H_2,G_2)$. If $\{\tilde{\bm{H}}^{(n)}_1\}_{n=1,2,\ldots}$ weakly converge to $\tilde{\bm{H}}_2$ compactly, then $\mathcal{M}_{1L}$ is said to be compactly dense in $\mathcal{M}_{2L}$.
\end{definition}

The above definition of denseness means that any mixed effects models in the class $\mathcal{M}_{2L}$ can be represented or approximated arbitrarily well by the models within another class $\mathcal{M}_{1L}$, in the sense of weak convergence of joint distributions $\tilde{H}(\bm{y}|\bm{x})$ across all $N$ observations. Less technically speaking,  $\mathcal{M}_{1L}$ can be interpreted as a model class that is at least as rich or flexible as the class $\mathcal{M}_{2L}$.


With all relevant definitions constructed, we now consider a class of generalized mixed effects models $\mathcal{M}_L^{\text{gen}}(\mathcal{X}):=\mathcal{M}_L(\mathcal{X};\mathcal{T}^{\text{gen}},\mathcal{H}^{\text{gen}},\mathcal{G}^{\text{gen}})$ expressed in the form of Equation (\ref{eq:gen_joint_cross}), where $\mathcal{T}^{\text{gen}}$ (also denoted as $\mathcal{T}^{\text{gen}}:=\mathcal{T}^{\text{gen}}_1\times\ldots\times\mathcal{T}^{\text{gen}}_L$), $\mathcal{H}^{\text{gen}}$ and $\mathcal{G}^{\text{gen}}$ all correspond to collections of any spaces or functions satisfying the following two mild technical assumptions: 

\begin{assumption} \label{asm:metric}
Each space $\bm{\Theta}_l\in\mathcal{T}^{\text{gen}}_l$ is equipped by a complete separable metric $d_{\bm{\Theta}_l}$.
\end{assumption}

\begin{assumption} \label{asm:continuity}
For every probability distribution functions $H\in\mathcal{H}^{\text{gen}}$, $H(\bm{y}_i|\bm{x}_i,\bm{\theta}_i)$ is continuous with respect to $(\bm{y}_i,\bm{x}_i,\bm{\theta}_i)$.
\end{assumption}


Now, we turn to the class of MMoE with a pre-determined choice of expert function $F_0$ and joint distribution of random effects $\Phi$. We express the class of MMoE as $\mathcal{M}_L^{\text{MMoE}}(\mathcal{X};F_0,\Phi):=\mathcal{M}_L(\mathcal{X};\mathcal{T}^{\text{MMoE}},\mathcal{H}^{\text{MMoE}},\mathcal{G}^{\text{MMoE}})$, where the two sets $\mathcal{T}_l^{\text{MMoE}}=\{\mathbb{R}\}$ and $\mathcal{G}^{\text{MMoE}}=\{\Phi\}$ only contain a single element each. Also, the set of mixed effects distributions is given by $\mathcal{H}^{\text{MMoE}}=\{F(\cdot;\bm{\alpha},\bm{\beta},\bm{\psi},g|\cdot,\cdot):\bm{\alpha}\in\mathcal{A},\bm{\beta}\in\mathcal{B},\bm{\psi}\in\bm{\Psi},g\in\mathbb{N}\}$, where $F$ is given by the form of Equation (\ref{eq:moe_disn_cross}). We also impose the following two restrictions on the choices of $F_0$ and $\Phi$, which are essential for the denseness property of the class of MMoE models:

\begin{assumption} \label{asm:denseness_condition}
$F_0$ satisfies the denseness condition outlined by Proposition 3.1 of \cite{fung2019moe1}, meaning that for every $\bm{q}\in \mathbb{R}^K$, there exists a sequence of parameters $\{\bm{\psi}^{(n)}(\bm{q})\}_{n=1,2,\ldots}$ such that $F_0(\cdot,\bm{\psi}^{(n)}(\bm{q}))\xrightarrow{\mathcal{D}}\bm{q}$ as $n\xrightarrow{}\infty$.
\end{assumption}

\begin{assumption} \label{asm:cont_random_effect_disn}
$\Phi_l$ is a continuous distribution function for every $l=1,\ldots,L$.
\end{assumption}

The above two assumptions are not necessarily mild. As discussed in \cite{fung2019moe1}, some common distributions, such as Pareto and exponential distributions, do not satisfy the denseness condition under Assumption \ref{asm:denseness_condition}. Moreover, Assumption \ref{asm:cont_random_effect_disn} does not hold whenever we choose any discrete distributions for $\Phi_l$. Nonetheless, note that the expert function $F_0$ and random effect distribution $\Phi_l$ are both pre-determined, so we have the control to choose suitable functions that fulfill Assumptions \ref{asm:denseness_condition} and \ref{asm:cont_random_effect_disn} before modelling a multilevel dataset via the MMoE. For example, one may choose Gamma, Weibull, log-normal, or inverse-Burr distributions as the expert function $F_0$ (\cite{fung2019moe1}), and select a normal distribution as the random effect distribution $\Phi_l$. 


We now present the main result which justifies the representation capability of the proposed class of MMoE models. The proof appears in the Appendix.

\begin{theorem} \label{thm:denseness_cross}
Suppose that Assumptions \ref{asm:indep} to \ref{asm:cont_random_effect_disn} are satisfied. Then, $\mathcal{M}_L^{\text{MMoE}}(\mathcal{X};F_0,\Phi)$ is compactly dense in $\mathcal{M}_L^{\text{gen}}(\mathcal{X})$.
\end{theorem}

\begin{remark}
The theorem above requires that the target distribution $H$ is a continuous distribution with respect to $\bm{y}_i$ (see Assumption \ref{asm:continuity}). However, it is also important to investigate into the approximation results for discrete distributions (see e.g. \cite{jiang1999hierarchical} and \cite{fung2019moe1}). As discussed by \cite{norets2010approximation}, any discrete distribution can be represented by a continuous latent distribution. Then, it is obvious that Theorem \ref{thm:denseness_cross} still holds for discrete distributions if the denseness condition in Assumption \ref{asm:denseness_condition} is changed from $\bm{q}\in\mathbb{R}^K$ to $\bm{q}\in\mathbb{N}^K$.
\end{remark}

Clearly, the denseness property provides a theoretical justification for the flexibility of the MMoE to simultaneously capture various model characteristics, including the joint distributions (e.g., multimodal distributions and dependence among outputs), the regression patterns (e.g., non-linear links and interactions among covariates), the random intercept (e.g., peculiar effects of the unobserved factors to the outputs) and the random slopes (e.g., interactions between covariates and random effects). As a result, the denseness property theoretically demonstrates the ability of the MMoE to potentially fit and resemble well any multilevel data which can be generated among a rich class of mixed effects models. Also, since our proposed model has a reduced structure (explained in the previous section), the denseness property also suggests that the proposed model is parsimonious.

\section{Nested mixed effects models for hierarchical data} \label{sec:nested}

Nested mixed effect model is a special case of the model described in Section \ref{sec:gen} with multiple random effects. Specifically, the $L$ random effects are levelled in a way that the first and $L$-th random effects correspond to the highest and the lowest levels respectively. Any observations sharing the same lower level random effect unit must also have the same upper level unit. If the random effects are not nested, the corresponding mixed effect model is called a ``crossed" mixed effect model. The school problem (see, e.g., \cite{aitkin1986statistical}) is a classical example where the random effects are nested. In this example, there are two factors affecting a student's performance: School and classroom. Since classmates must come from the same school, we have $L=2$ with ``school" and ``classroom" respectively being the first and second levels of random effects.


Before defining a nested model, it is more convenient to adopt an alternative set of notations to identify the observations. Denote $\bm{i}=(i_1,i_2,\ldots,i_{L+1})$ as an identifier of an observation, and $\bm{i}_l=(i_1,i_2,\ldots,i_l)$ as an identifier up to level $l$. Here, $i_1$ is a label of the level-$1$ factor unit. For $l=2,3,\ldots,L$, $i_l$ is a label of the level-$l$ factor unit given that the labels of the first $(l-1)$ factors levels are $\bm{i}_{l-1}$. $i_{L+1}$ represents the observation label given that the $L$ factors are labelled by $\bm{i}_L$. For example, in the school problem with $L=2$, $\bm{i}=(2,3,5)$ corresponds to the fifth student of the third classroom of the second school. 


Furthermore, denote $N_0$ as the number of level-$1$ factor units, so that the support of $i_1$ is given by $\mathcal{I}_1:=\{1,\ldots,N_0\}$. For $l=2,3,\ldots,L$, define $N_{\bm{i}_{l-1}}$ as the number of level-$l$ factor units where the first $(l-1)$ factor levels are labelled as $\bm{i}_{l-1}$, such that the support of $\bm{i}_l$ is $\mathcal{I}_l:=\{\bm{i}_l:\bm{i}_{l-1}\in\mathcal{I}_{l-1},i_l=1,\ldots,N_{\bm{i}_{l-1}}\}$. Similarly, $N_{\bm{i}_L}$ is the number of observations having $\bm{i}_L$ as the label of the $L$ factors. Then, the support of $\bm{i}$ is $\mathcal{I}:=\{\bm{i}:\bm{i}_L\in\mathcal{I}_L,i_{L+1}=1,\ldots,N_{\bm{i}_L}\}$. Also, the total number of observations is given by $N=\sum_{i_1=1}^{N_0}\sum_{i_2=1}^{N_{\bm{i}_{1}}}\cdots\sum_{i_L=1}^{N_{\bm{i}_{L-1}}}N_{\bm{i}_L}$. A tree diagram in Figure \ref{fig:hierarchical} visualizes the structure of nested hierarchical data.

As a result, the nested multilevel dataset is given by $(\bm{y},\bm{x}):=\{(\bm{y}_{\bm{i}},\bm{x}_{\bm{i}})\}_{\bm{i}\in\mathcal{I}}$, where $\bm{x}_{\bm{i}}\in\mathcal{X}\subseteq\mathbb{R}^P$ and $\bm{y}_{\bm{i}}\in\mathcal{Y}\subseteq\mathbb{R}^K$ are respectively the input and output of an observation labelled as $\bm{i}\in\mathcal{I}$. 


In this section, we will define the generalized class of nested mixed effects models, formulate the proposed MMoE model with nested random effects, and construct the denseness theory for the class of nested MMoE.

\afterpage{
\begin{figure}[!h]
\centering
\includegraphics[width=\linewidth]{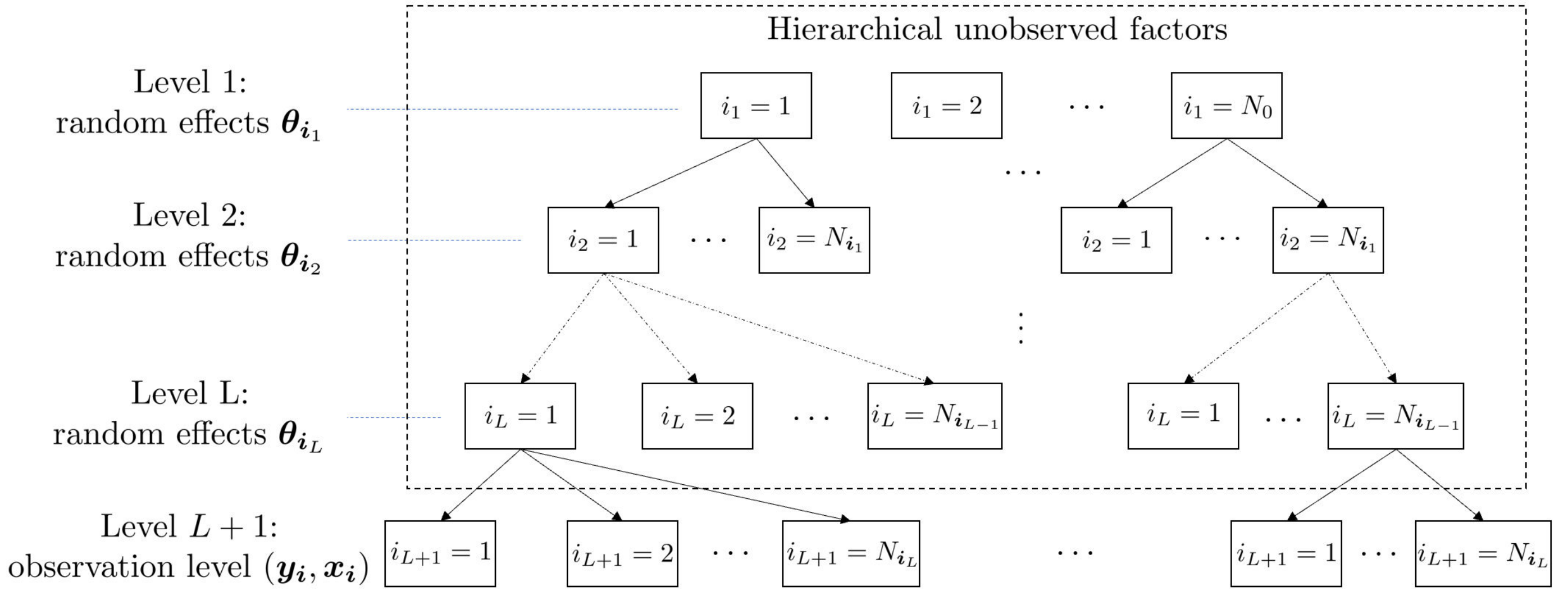}
\caption{Hierarchical data structure and modelling framework.}
\label{fig:hierarchical}
\end{figure}
}

\subsection{Generalized nested mixed effect model}

Similar to the MMoE defined in Section \ref{sec:gen}, under the nested MMoE, the response $\bm{y}_{\bm{i}}$ depends on its covariates $\bm{x}_{\bm{i}}$ and $L$ levels of latent random factors $\bm{\theta}_{\bm{i}_1},\ldots,\bm{\theta}_{\bm{i}_L}$ illustrated in Figure \ref{fig:hierarchical}. The dependence assumption among the latent factors is stated as follows:

\begin{assumption} \label{asm:dep_nested}
Conditioned on $\bm{\theta}_{\bm{i}_1},\ldots,\bm{\theta}_{\bm{i}_L}$, the $N$ observations are independent. The set of parents of $\bm{\theta}_{\bm{i}_l}$ is given by $\text{pa}(\bm{\theta}_{\bm{i}_l})=(\bm{\theta}_{\bm{i}_1},\ldots,\bm{\theta}_{\bm{i}_{(l-1)}})$ for $l=1,\ldots,L$.
\end{assumption}

In other words, the lower level factor only depends directly on its corresponding upper level factors. Under a nested hierarchical data structure, we are able to relax the independence assumption in Assumption \ref{asm:indep} by allowing the dependence of lower level random effects on their parents (upper level effects). For $l=1,\ldots,L$, we also denote $G_l$ as the distribution of the level-$l$ factor conditioned on $\text{pa}(\bm{\theta}_{\bm{i}_l})$.


The joint distribution of $\bm{y}$ given $\bm{x}$ is given by
\begin{equation} \label{eq:gen_joint_nested}
\tilde{H}(\bm{y}|\bm{x})=\int_{\tilde{\bm{\Theta}}}\left[\prod_{\bm{i}\in\mathcal{I}}H(\bm{y}_{\bm{i}};\bm{x}_{\bm{i}}|\bm{\theta}_{\bm{i}})\right]dG(\tilde{\bm{\theta}})
\end{equation}
with
\begin{equation} \label{eq:gen_latent_nested}
G(\tilde{\bm{\theta}})=\prod_{i_1=1}^{N_0}G_1(\bm{\theta}_{\bm{i}_1})\prod_{i_2=1}^{N_{\bm{i}_1}}G_2(\bm{\theta}_{\bm{i}_2}|\bm{\theta}_{\bm{i}_1})\cdots\prod_{i_{L}=1}^{N_{\bm{i}_{L-1}}}G_{L}(\bm{\theta}_{\bm{i}_{L}}|\bm{\theta}_{\bm{i}_1},\ldots,\bm{\theta}_{\bm{i}_{L-1}}),
\end{equation}
where $\tilde{\bm{\theta}}=\{\bm{\theta}_{\bm{i}_l}:\bm{i}_l\in\mathcal{I}_l,l=1,\ldots,L\}$.


\subsection{Nested mixed mixture of experts model}
Analogous to the MMoE introduced in Section \ref{sec:moe}, we construct a nested MMoE for hierarchical data. Denote $\bm{w}_{\bm{i}}=(w_{\bm{i}_1},\ldots,w_{\bm{i}_L})\in\mathbb{R}^{L}$ as the $L$ levels of random effects of observation $\bm{i}$. Also, define $\bm{w}=\{w_{\bm{i}_l}\}_{\bm{i}_l\in\mathcal{I}_l,l=1,\ldots,L}$ as all the random effects aggregated across all observations. With a slight change of notations from Equations (\ref{eq:moe_disn_cross}) and (\ref{eq:moe_gate_cross}),the distribution function of $\bm{y}_{\bm{i}}$ conditional on $\bm{x}_{\bm{i}}$ and $\bm{w}_{\bm{i}}$ is then
\begin{equation} \label{eq:moe_disn_nested}
F(\bm{y}_{\bm{i}}; \bm{\alpha}, \bm{\beta}, \bm{\psi}, g| \bm{x}_{\bm{i}}, \bm{w}_{\bm{i}}) = \sum_{j=1}^{g} \pi_j(\bm{x_{\bm{i}}}, \bm{w}_{\bm{i}}; \bm{\alpha}, \bm{\beta}) F_0(\bm{y}_{\bm{i}}; \bm{\psi}_j),
\end{equation}
where the gating function given by
\begin{equation}
    \pi_j(\bm{x}_{\bm{i}},\bm{w}_{\bm{i}}; \bm{\alpha}, \bm{\beta}) = \frac{\exp\{\alpha_{j0} + \bm{\alpha}_j^T \bm{x}_{\bm{i}} +  \bm{\beta}_j^T \bm{w}_{\bm{i}}\}}{\sum_{j'=1}^{g}\exp\{\alpha_{j'0} + \bm{\alpha}_{j'}^T \bm{x}_{\bm{i}} +  \bm{\beta}_{j'}^T \bm{w}_{\bm{i}}\}}, \quad j = 1, 2, \dots, g.
\end{equation}

Similar to Section \ref{sec:moe}, the random effects $\{w_{\bm{i}_l}\}_{l=1,\ldots,L;i_l=1,\ldots,N_{\bm{i}_{l-1}}}$ are constructed to be independent across $l$ and $i_l$ with $w_{\bm{i}_l}\sim \Phi_l$. This construction is simplified from Assumption \ref{asm:dep_nested} of the generalized nested mixed models where the random effects may depend on their parents. Adapting from Equations (\ref{eq:moe_joint_cross}) and (\ref{eq:moe_phi_cross}), the joint distribution of $\bm{y}$ given $\bm{x}$ is

\begin{equation} \label{eq:moe_joint_nested}
\tilde{F}(\bm{y}; \bm{x}):=\tilde{F}(\bm{y}; \bm{\alpha}, \bm{\beta}, \bm{\Psi}, g| \bm{x})=\int\prod_{i\in\mathcal{I}} F(\bm{y}_{\bm{i}}; \bm{\alpha}, \bm{\beta}, \bm{\Psi}| \bm{x}_{\bm{i}}, \bm{w}_{\bm{i}})d\Phi(\bm{w}),
\end{equation}
where $\Phi$ is the joint distribution of $\bm{w}$ given by

\begin{equation} \label{eq:moe_phi_nested}
\Phi(\bm{w})=\prod_{l=1}^{L}\prod_{\bm{i}_l\in\mathcal{I}_l}\Phi_l(w_{\bm{i}_l})
\qquad
\text{or}
\qquad
d\Phi(\bm{w})=\prod_{l=1}^{L}\prod_{\bm{i}_l\in\mathcal{I}_l}\Phi_l(d w_{\bm{i}_l}).
\end{equation}

From Equation (\ref{eq:moe_phi_nested}) above, the random effects are still assumed to be independent under the nested MMoE. However, we will show in the following subsection that such a specification suffices to approximate the dependence of lower level random effects on their parents under a hierarchical data structure.


\subsection{Denseness theory for nested MMoE}
\sloppy Analogous to Section \ref{sec:dense}, it is desirable to develop an approximation theory for the nested MMoE in the space of the generalized nested mixed effect models. Denote $\bm{N}=(N_0,\{N_{\bm{i}_1}\}_{\bm{i}_1\in\mathcal{I}_1},\{N_{\bm{i}_2}\}_{\bm{i}_2\in\mathcal{I}_2},\ldots,\{N_{\bm{i}_L}\}_{\bm{i}_L\in\mathcal{I}_L})$ as the number of factors belonging to each parent factors for each level with $N_0\in\mathbb{N}$ and $N_{\bm{i}_l}\in\mathbb{N}$ for $l=1,\ldots,L$, and $\mathcal{N}$ contains all combinations of possible $\bm{N}$. Other notations, unless specified otherwise, are consistent to those defined by Section \ref{sec:dense}. The equivalent definitions analogous to Section \ref{sec:dense} for hierarchical data structure are listed as follows:

\begin{definition}
A class of nested mixed effects models $\mathcal{M}_L(\mathcal{X};\mathcal{T},\mathcal{H},\mathcal{G}):=\{\tilde{H}(\cdot;\mathcal{X};\bm{\Theta},H,G):\bm{\Theta}_l\in\mathcal{T}_l,H\in\mathcal{H},G_l\in\mathcal{G}_l,l=1,\ldots,L\}$ is a collection of nested mixed effects distributions $\tilde{H}(\cdot;\mathcal{X};\bm{\Theta},H,G)$, where each nested mixed effects distribution $\tilde{H}(\cdot;\mathcal{X};\bm{\Theta},H,G):=\{\tilde{H}(\bm{y}|\bm{x}):=\tilde{H}(\bm{y}|\bm{x};\bm{\Theta},H,G)=\int_{\tilde{\bm{\Theta}}}\left[\prod_{\bm{i}\in\mathcal{I}}H(\bm{y}_i|\bm{x}_{\bm{i}},\bm{\theta}_{\bm{i}})\right]dG(\tilde{\bm{\theta}}):\bm{x}_{\bm{i}}\in\mathcal{X},\bm{i}\in\mathcal{I},\bm{N}\in\mathcal{N}\}$ is itself a collection of joint probability distributions.
\end{definition}

\begin{definition}
Consider a sequence of nested mixed effects distributions $\tilde{\bm{H}}^{(n)}:=\tilde{H}(\cdot;\mathcal{X};\bm{\Theta}^{(n)},H^{(n)},G^{(n)})$ and a target nested mixed effects distribution $\tilde{\bm{H}}:=\tilde{H}(\cdot;\mathcal{X};\bm{\Theta},H,G)$. We say that $\{\tilde{\bm{H}}^{(n)}\}_{n=1,2,\ldots}$ weakly converge to $\tilde{\bm{H}}$ if and only if for every given $\bm{x}_{\bm{i}}\in\mathcal{X}$ (for all ${\bm{i}}\in\mathcal{I}$), $\bm{N}\in\mathcal{N}$, we have $\tilde{H}(\cdot|\bm{x};\bm{\Theta}^{(n)},H^{(n)},G^{(n)})\xrightarrow{\mathcal{D}}\tilde{H}(\cdot|\bm{x};\bm{\Theta},H,G)$ as $n\rightarrow\infty$, where $\xrightarrow{\mathcal{D}}$ represents a weak convergence or convergence in distribution. If the distributional convergence is uniform across any compact input space $\bar{\mathcal{X}}\subseteq\mathcal{X}$, i.e. $\tilde{H}(\bm{y}|\bm{x};\bm{\Theta}^{(n)},H^{(n)},G^{(n)})\rightarrow\tilde{H}(\bm{y}|\bm{x};\bm{\Theta},H,G)$ uniformly on $(\bm{y},\bm{x})$ with $\bm{x}_i\in\bar{\mathcal{X}}$ (for all $i\in\{1,\ldots,N\}$), then we say that $\{\tilde{\bm{H}}^{(n)}\}_{n=1,2,\ldots}$ weakly converge to $\tilde{\bm{H}}$ compactly.
\end{definition}

The denseness definition for hierarchical data structure (nested mixed effects models) is exactly the same as Definition \ref{def:dense}. Define $\mathcal{M}_L^{*\text{gen}}(\mathcal{X})$ as the class of generalized nested mixed effects models expressed in Equation (\ref{eq:gen_joint_nested}), subject to Assumptions \ref{asm:metric} and \ref{asm:continuity}. Also denote $\mathcal{M}_L^{*\text{MMoE}}(\mathcal{X};F_0,\Phi)$ as the class of nested MMoE given by Equation (\ref{eq:moe_joint_nested}). We have the following denseness theorem:

\begin{theorem} \label{thm:denseness_nested}
Suppose that Assumptions \ref{asm:metric} to \ref{asm:dep_nested} hold. Then, $\mathcal{M}_L^{*\text{MMoE}}(\mathcal{X};F_0,\Phi)$ is compactly dense in $\mathcal{M}_L^{*\text{gen}}(\mathcal{X})$.
\end{theorem}

The proof is leveraged to Appendix \ref{sec:apx2}. Theorem \ref{thm:denseness_nested} suggests that the nested MMoE has a potential to approximate any generalized nested mixed effect models arbitrarily accurately, even if the random effects are restricted to be independent under the nested MMoE while the random effects under the generalized nested mixed effect models can be dependent (Assumption \ref{asm:dep_nested}). In contrast, under Theorem \ref{thm:denseness_cross}, the MMoE can only approximate mixed effect models with independent random effects (Assumption \ref{asm:indep}). Therefore, given that the data structure is hierarchical, Theorem \ref{thm:denseness_nested} is a stronger theoretical result than Theorem \ref{thm:denseness_cross}.

\section{Numerical Illustration} \label{sec:numerical}

In this section, we present a numerical example to demonstrate how the proposed MMoE model can approximate any mixed effects model. For illustration purposes, all variables in this section will be one-dimensional, i.e.~$y_i$ for the response, $\theta_1$ for the general random effect, and $w_1$ for the random effect used in MMoE. We assume there are no covariates $\bm{x}_i$ for reasons to be explained later. In the following, we first demonstrate how to construct an MMoE such that the conditional distribution of $y_i$ given $\theta_1$ can be approximated arbitrarily well by that of $y_i$ given $w_1$ for a single observation $y_i$. Then, we discuss how this is related to the denseness property of MMoE in Theorem \ref{thm:denseness_cross}.

We first construct the true model as follows. Consider $\tilde{\bm{\Theta}} = \bm{\Theta}_1 = [-2, 2]$ as the space for $\theta_1$. The general random effect $\theta_1$ is assumed to follow $\textrm{Uniform}(-2, 2)$ with distribution function $G_1(\theta_1) = (2+\theta_1)/4$ and density $g_1(\theta_1) = 1/4$. We assume $y_i | \theta_1 \sim \textrm{Normal}(\theta_1, 1)$ with density $h(y_i|\theta_1) = \frac{1}{\sqrt{2\pi}}e^{-\frac{1}{2}(y_i-\theta_1)^2}$. 

We then aim to use MMoE to approximate the true model mentioned above. For MMoE, we choose $\mathcal{W}_{1} = \mathbb{R}$ as the space for $w_1$. The random effect $w_1$ in MMoE is assumed to be standard normal with distribution function $\Phi_1(\cdot)$ and density $\phi_1(\cdot)$. Denote $f(y_i|w_1)$ as the conditional density of $y_i$ given $w_1$. We aim to construct $f(y_i|w_1)$ such that the conditional distribution $h(y_i|\theta_1)$, $\forall \theta_1 \in \bm{\Theta}_1$, can be approximated arbitrarily well by $f(y_i|w_1)$, $\exists w_1 \in \mathcal{W}_1$.

We start by discretizing both $\bm{\Theta}_1$ and $\mathcal{W}_1$ into $D_1$ pairs of subspaces $\{\bm{\Theta}_{1, d_1}\}_{d_1=1,\ldots,D_1}$ and $\{ \mathcal{W}_{1, d_1}\}_{d_1=1,\ldots,D_1}$ such that $G_1(\bm{\Theta}_{1, d_1}) = \Phi_1(\mathcal{W}_{1, d_1})$ for $d_1=1, 2, \dots, D_1$. For each $\bm{\Theta}_{1, d_1}$, we pick $\theta^{*}_{1, d_1} \in \bm{\Theta}_{1, d_1}$ such that $h(y_1|\theta^{*}_{1, d_1})$ is a reasonable approximation of $h(y_1|\theta_1 \in \bm{\Theta}_{1, d_1})$. Then, each $h(y_1|\theta^{*}_{1, d_1})$ is approximated by the following finite mixture
\begin{equation} \label{eq:numerical-approx-with-moe}
    f(y_i|w_1) = \sum_{d_1=1}^{D_1} \xi^{(u)}_{d_1}(w_1) h(y_1|\theta^{*}_{1, d_1})
\end{equation}
where the mixing weights are given by
\begin{equation} \label{eq:numerical-approx-with-moe-mixing-weights}
    \xi^{(u)}_{d_1}(w_1) = \frac{\exp\{ u(\tilde{\beta}_{d_1, 0} + \tilde{\beta}_{d_1, 1}w_1) \}}{\sum_{d'_1=1}^{D_1}\exp\{ u(\tilde{\beta}_{d'_1, 0} + \tilde{\beta}_{d'_1, 1}w_1) \}}, \quad d_1 = 1, 2, \dots, D_1.
\end{equation}

Note Equation (\ref{eq:numerical-approx-with-moe}) is not yet the formulation of MMoE, since each $h(y_i|\theta^{*}_{1, d_1})$ is still dependent on $\theta^{*}_{1, d_1}$. Following the same idea in \cite{fung2019moe1}, Appendix \ref{sec:apx1} shows that $h(y_i|\theta^{*}_{1, d_1})$ can be further approximated by a finite mixture of expert functions with fixed parameters independent of the random effects, which ultimately leads to an MMoE similar to Equation (\ref{eq:moe_disn_cross}). We will omit such approximation procedures and work with Equation (\ref{eq:numerical-approx-with-moe}) to avoid further complications which are not relevant to random effects. Meanwhile, the approximation procedures in the presence of covariates $\bm{x}_i$ are also similar to those presented in \cite{fung2019moe1} and are therefore deferred to Appendix \ref{sec:apx1}, which is why we have presented an example without covariates for conciseness.

Apart from $h(y_i|\theta^{*}_{1, d_1})$, the mixing weights $\xi^{(u)}_{d_1}(w_1)$ in Equation (\ref{eq:numerical-approx-with-moe-mixing-weights}) contains an additional control parameter $u$ compared with MMoE. When the coefficients $\{(\tilde{\beta}_{d_1, 0}, \tilde{\beta}_{d_1, 1})\}_{d_1=1, 2, \dots, D_1}$ are carefully constructed, $\xi^{(u)}_{d_1}(w_1) \rightarrow 1\{w_1 \in \mathcal{W}_{1, d_1}\}$ as $u \rightarrow \infty$ for $d_1=1, 2, \dots, D_1$. Consequently, the finite mixture in Equation (\ref{eq:numerical-approx-with-moe}) degenerates to $h(y_i|\theta^{*}_{1, d_1})$ whenever $w_1 \in \mathcal{W}_{1, d_1}$ and $u \rightarrow \infty$. 

To summarize, given $\bm{\Theta}_{1, d_1}$, we first approximate $h(y_i|\theta_1 \in \bm{\Theta}_{1, d_1})$ by $h(y_i|\theta^{*}_{1, d_1})$ for some $\theta^{*}_{1, d_1} \in \bm{\Theta}_{1, d_1}$. Then, $h(y_i|\theta^{*}_{1, d_1})$ is further approximated by $f(y_i|w_1 \in \mathcal{W}_{1, d_1})$, which is done for all pairs of $(\bm{\Theta}_{1, d_1}, \mathcal{W}_{1, d_1})$. In effect, the subspace $\mathcal{W}_{1, d_1}$ of random effects $w_1$ in MMoE becomes specialised in approximating the subspace $\bm{\Theta}_{1, d_1}$ of some general random effects $\theta_1$. As the partition of $(\bm{\Theta}_{1, d_1}, \mathcal{W}_{1, d_1})$ becomes finer and finer, we can find a mapping $\bm{\Theta}_1\rightarrow\mathcal{W}_1$ such that the conditional distribution $h(y_i|\theta_1)$ for any $\theta_1 \in \bm{\Theta}_1$ can be approximated by $f(y_i|w_1)$ for some $w_1 \in \mathcal{W}_1$.

As a concrete example, Table \ref{table:numerical-demonstration} illustrates the above procedure with $D_1=5$ partitions for $\bm{\Theta}_1$ and $\mathcal{W}_{1}$, where we choose the midpoint $\theta^{*}_{1, d_1}$ as the representative value for the subspace $\bm{\Theta}_{1, d_1}$. The values of $\{(\tilde{\beta}_{d_1, 0}, \tilde{\beta}_{d_1, 1})\}_{d_1=1, 2, \dots, D_1}$ are chosen according to Lemma 3.1 in \cite{fung2019moe1} to ensure the convergence of mixing weights $\xi^{(u)}_{d_1}(w_1)$ to the degenerate indicators $1\{w_1 \in \mathcal{W}_{1, d_1}\}$. This example is illustrated by the second row in Figure \ref{fig:numerical-example}, where the each vertical slice of the plot shows the conditional distribution of $f(y_i|w_1)$ (blue) for $u=1, 10, 100$ and $1000$ from left to right, and $h(y_i|\theta_1 \in \bm{\Theta}_{1, d_1})$ (green) for $d_1 = 1, 2, \dots, 5$. Indeed, we observe each $f(y_i|w_1 \in \mathcal{W}_{1, d_1})$ approximates the corresponding $h(y_i|\theta_1 \in \bm{\Theta}_{1, d_1})$ better as $u$ increases. The third and fourth row of Figure \ref{fig:numerical-example} show the same sets of plots, but for $D_1 = 10$ and $D_1 = 100$ partitions for both $\bm{\Theta}_{1}$ and $\mathcal{W}_1$. However, if the number of partition is too small, the MMoE will underestimate the variance of $y_i$ given $w_1$ compared with the variance of $y_i$ given $\theta_1$ under the true model, even if $u\rightarrow\infty$, as shown in the first row of Figure \ref{fig:numerical-example} with $D_1=2$. In the limiting case where the partition becomes infinitely fine, we can find a unique $f(y_i|w_1)$ to approximate $h(y_i|\theta_1)$ for all $\theta_1 \in \bm{\Theta}_1$. For example, the conditional distribution $h(y_i|\theta_1 = -1)$ in the original mixed effects model will be eventually be approximated by $f(y_i|w_1 = -0.67)$ in MMoE.

Finally, once the conditional distribution $h(y_i|\theta_1)$ can be approximated by $f(y_i|w_1)$, the denseness property in Theorem \ref{thm:denseness_cross} naturally follows. More specifically, given the conditional independence assumption in Equation (\ref{eq:gen_y_conditional}), the conditional joint distribution of $\prod_{i=1}^{N} h(y_i|\theta_i)$ can also be approximated arbitrarily well by $\prod_{i=1}^{N} f(y_i|w_i)$, where $\theta_i$ and $w_i$ for each observation $y_i$ are specified by $\theta_1$, $w_1$ and the mapping function $c_1(i)$. In the case of finite partition of subspaces $\{\bm{\Theta}_{1, d_1}\}_{d_1=1,\ldots,D_1}$ and $\{ \mathcal{W}_{1, d_1}\}_{d_1=1,\ldots,D_1}$ such that $G_1(\bm{\Theta}_{1, d_1}) = \Phi_1(\mathcal{W}_{1, d_1})$ for $d_1=1, 2, \dots, D_1$, the marginal joint distribution of all $y_i$'s are also approximated arbitrarily well, which is evident from
\begin{equation}
    f(\bm{Y}) = \sum_{d_1=1}^{D_1} \left[\prod_{i=1}^{N} h(y_i|\theta_i) 1\{\theta_i \in \bm{\Theta}_{1, d_1} \}\right] G_1(\bm{\Theta}_{1, d_1})
\end{equation}
and
\begin{equation}
    g(\bm{Y}) = \sum_{d_1=1}^{D_1} \left[\prod_{i=1}^{N} f(y_i|w_i) 1\{w_i \in \mathcal{W}_{1, d_1} \} \right] \Phi_1(\mathcal{W}_{1, d_1}).
\end{equation}
In the limiting scenario where the partition of $\bm{\Theta}_1$ and $\mathcal{W}_1$ becomes infinitely fine, the above equations recover the integral forms in Equation (\ref{eq:gen_joint_cross2}) and (\ref{eq:moe_joint_cross}), thus demonstrating the denseness property of MMoE in Theorem \ref{thm:denseness_cross}.

\afterpage{
    
\renewcommand{\arraystretch}{1.5}

\begin{table}[!ht]
\centering
\caption{Example of approximation with five partitions of $\bm{\Theta}_{1}$ and $\mathcal{W}_{1}$.}
\label{table:numerical-demonstration}

\begin{tabular*}{\textwidth}{c@{\extracolsep{\fill}}c@{\extracolsep{\fill}}c@{\extracolsep{\fill}}c@{\extracolsep{\fill}}c@{\extracolsep{\fill}}c@{\extracolsep{\fill}}}
    \hline \hline
$d_1$        & 1               & 2                & 3               & 4              & 5            \\ \hline
$\bm{\Theta}_{1,d_1}$       & $[-2.00, -1.20]$  & $(-1.20, -0.40]$   & $(-0.40, 0.40]$   & $(0.40, 1.20]$   & $(1.20, 2.00]$   \\
$\mathcal{W}_{1,d_1}$           & $(-\infty, -0.84]$ & $(-0.84, -0.25]$ & $(-0.25, 0.25]$ & $(0.25, 0.84]$ & $(0.84, +\infty)$ \\
$\theta^{*}_{1, d_1}$ & $-1.60$            & $-0.80$             & $0.00$               & $0.80$            & $1.60$          \\
$\tilde{\beta}_{d_1, 0}$     & $0.00$               & $0.23$           & $0.29$          & $0.23$         & $0.00$            \\
$\tilde{\beta}_{d_1, 1}$     & $0.00$               & $0.25$             & $0.50$             & $0.75$           & $1.00$      \\ \hline \hline
\end{tabular*}

\vspace{1em}

\end{table}

\renewcommand{\arraystretch}{1.0}

\begin{figure}[!h]
    \centering
    \caption{Illustration of using $f(y_i|w_1)$ (blue) to approximate $h(y_i|\theta_1 \in \bm{\Theta}_{1, d_1})$ (green) for different numbers of partitions of $\bm{\Theta}_{1}$ and $\mathcal{W}_{1}$. Each vertical slice corresponds to the conditional density of $f(y_i|w_1)$ or $h(y_i|\theta_1 \in \bm{\Theta}_{1, d_1})$. From left to right, the control parameter $u$ ranges from 1, 10, 100 to 1000 in the blue plots, and eventually $f(y_i|w_1) = h(y_i|\theta_1 \in \bm{\Theta}_{1, d_1})$ for all $w_1 \in \mathcal{W}_{1, d_1}$ as $u \rightarrow \infty$. From top to bottom, the number of partition $D_1$ ranges from 2, 5, 10 to 100. When the partition becomes infinitely fine and $u \rightarrow \infty$, each $h(y_i|\theta_1)$ is approximated arbitrarily well by some $f(y_i|w_1)$, e.g.~$h(y_i|\theta_1 = 1) = f(y_i|w_1 = 0.67)$.}
    \label{fig:numerical-example}
    
    \begin{subfigure}{0.18\textwidth}
        \includegraphics[width=\textwidth]{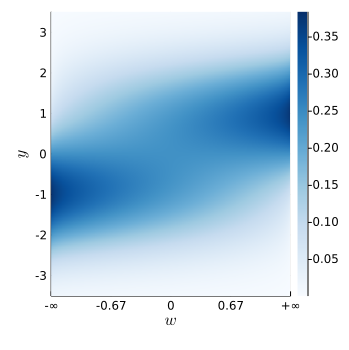}
    \end{subfigure}
    ~
    \begin{subfigure}{0.18\textwidth}
        \includegraphics[width=\textwidth]{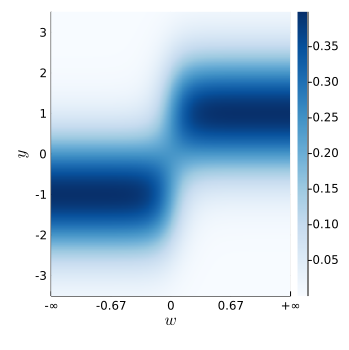}
    \end{subfigure}
    ~
    \begin{subfigure}{0.18\textwidth}
        \includegraphics[width=\textwidth]{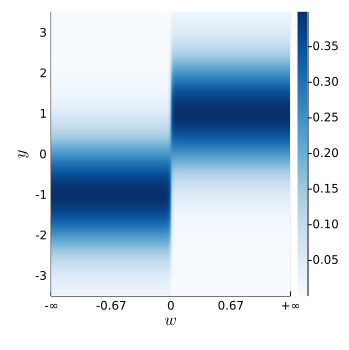}
    \end{subfigure}
    ~
    \begin{subfigure}{0.18\textwidth}
        \includegraphics[width=\textwidth]{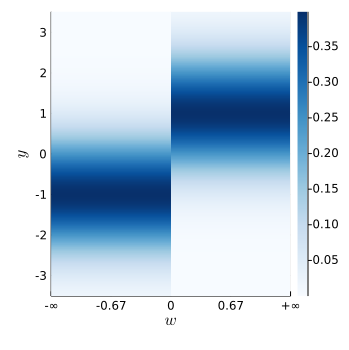}
    \end{subfigure}
    ~
    \begin{subfigure}{0.18\textwidth}
        \includegraphics[width=\textwidth]{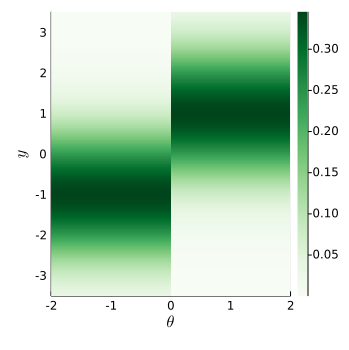}
    \end{subfigure}
    
    \begin{subfigure}{0.18\textwidth}
        \includegraphics[width=\textwidth]{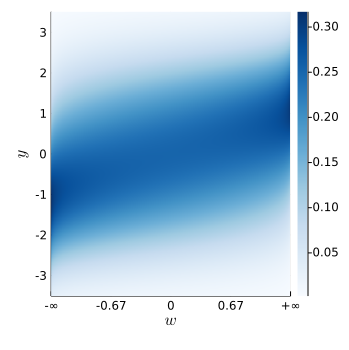}
    \end{subfigure}
    ~
    \begin{subfigure}{0.18\textwidth}
        \includegraphics[width=\textwidth]{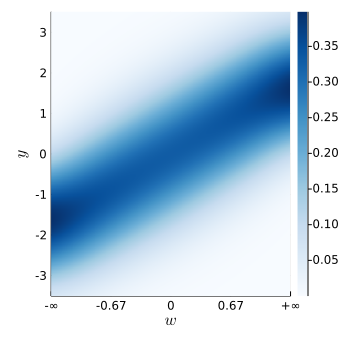}
    \end{subfigure}
    ~
    \begin{subfigure}{0.18\textwidth}
        \includegraphics[width=\textwidth]{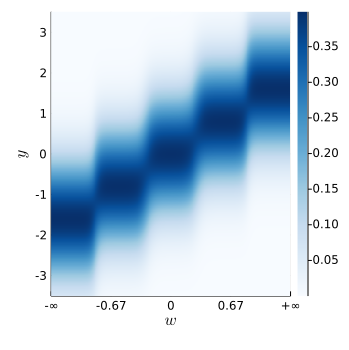}
    \end{subfigure}
    ~
    \begin{subfigure}{0.18\textwidth}
        \includegraphics[width=\textwidth]{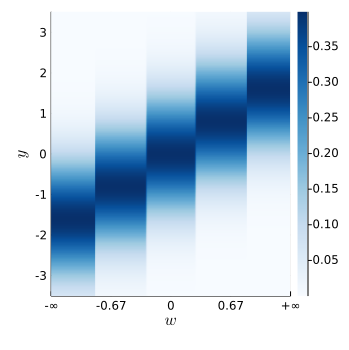}
    \end{subfigure}
    ~
    \begin{subfigure}{0.18\textwidth}
        \includegraphics[width=\textwidth]{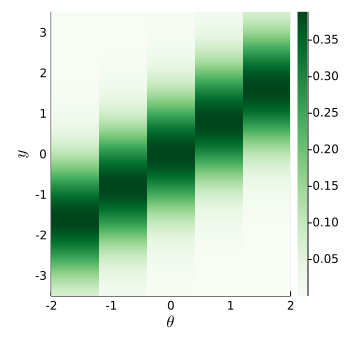}
    \end{subfigure}
    
    \begin{subfigure}{0.18\textwidth}
        \includegraphics[width=\textwidth]{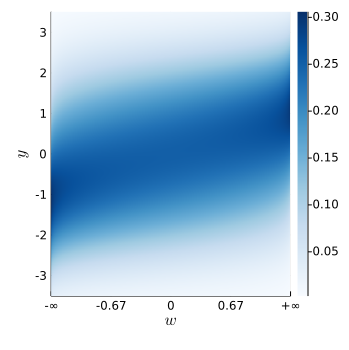}
    \end{subfigure}
    ~
    \begin{subfigure}{0.18\textwidth}
        \includegraphics[width=\textwidth]{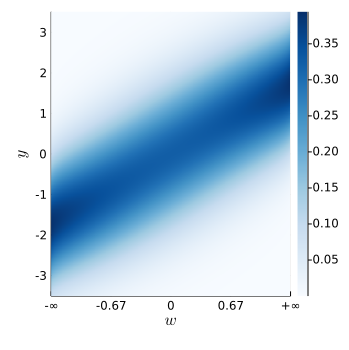}
    \end{subfigure}
    ~
    \begin{subfigure}{0.18\textwidth}
        \includegraphics[width=\textwidth]{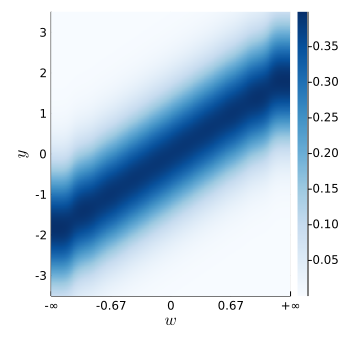}
    \end{subfigure}
    ~
    \begin{subfigure}{0.18\textwidth}
        \includegraphics[width=\textwidth]{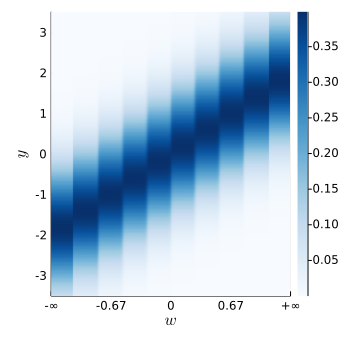}
    \end{subfigure}
    ~
    \begin{subfigure}{0.18\textwidth}
        \includegraphics[width=\textwidth]{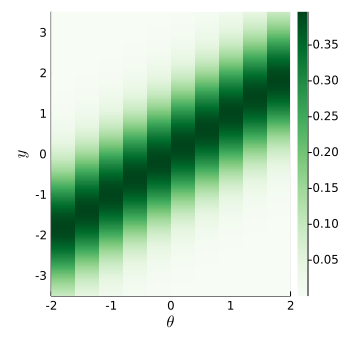}
    \end{subfigure}

    \begin{subfigure}{0.18\textwidth}
        \includegraphics[width=\textwidth]{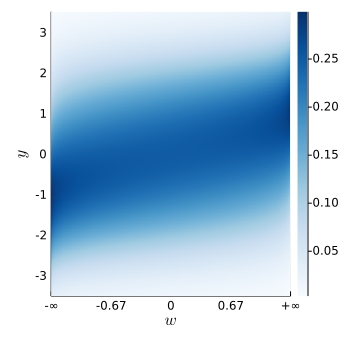}
    \end{subfigure}
    ~
    \begin{subfigure}{0.18\textwidth}
        \includegraphics[width=\textwidth]{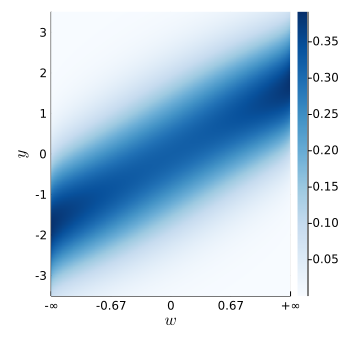}
    \end{subfigure}
    ~
    \begin{subfigure}{0.18\textwidth}
        \includegraphics[width=\textwidth]{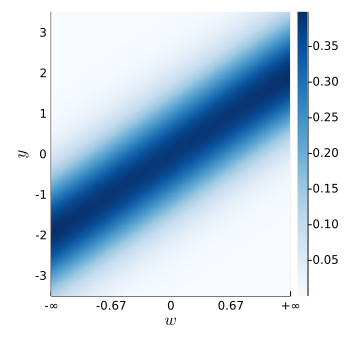}
    \end{subfigure}
    ~
    \begin{subfigure}{0.18\textwidth}
        \includegraphics[width=\textwidth]{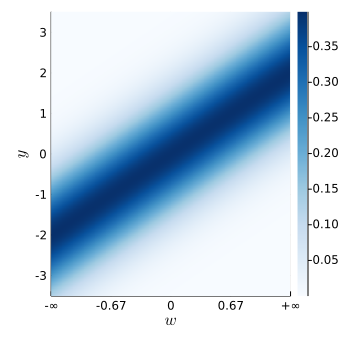}
    \end{subfigure}
    ~
    \begin{subfigure}{0.18\textwidth}
        \includegraphics[width=\textwidth]{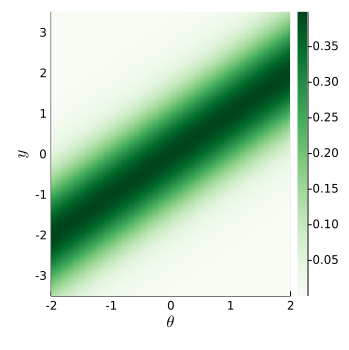}
    \end{subfigure}
    
\end{figure}

\pagebreak
}

\section{Discussions} \label{sec:discussions}
In this paper, we introduce a class of the mixed mixture of experts models (MMoE) for multilevel regression data. We prove that the MMoE is dense in the space of generalized mixed effects models, which is a rich class containing almost all models in the literature having independent random effects, in the sense of weak convergence. We further study a special case where the data is hierarchical in structure. In this case, the proposed nested MMoE is shown to be dense in the space of generalized nested mixed effects models where the random effects can possibly be dependent. The two denseness theorems justify the versatility of the MMoE in catering for a broad range of multilevel data characteristics, including the marginal distribution, dependence, regression link, random intercept and random slope.


This paper aims to prove the most general results imposing minimal assumptions. The only practical restriction is that the expert function $F_0(\bm{y}_i;\bm{\psi}_j)$ in Equation (\ref{eq:moe_disn_cross}) needs to approximate any degenerate distributions (Assumption \ref{asm:denseness_condition}). This assumption is much weaker than those applied to the existing approximation theorems (see, e.g, \cite{norets2014posterior} and \cite{nguyen2019approximation}), which require that the expert density function is a scalable symmetric function (Equation (3.1) of \cite{norets2014posterior}), and that the target density function does not change abruptly w.r.t. $\bm{y}_i$ and $\bm{x}_i$ (Equation (3.2) of \cite{norets2014posterior}). On the other hand, there are several limitations of the denseness theorems formulated in this paper. Firstly, weak convergence does not guarantee the approximation capability in terms of moments (e.g., the mean function studied by \cite{nguyen2016universal}) or some distance metrics (e.g., the KL divergence studied by \cite{jiang1999approximation}, \cite{norets2010approximation} and \cite{nguyen2019approximation}). To establish the denseness theorems w.r.t. moments and other distance metrics, one needs to further assume that the moments of the MMoE expert functions and the target distributions are finite, and that the conditions indicated by Equations (3.1) and (3.2) of \cite{norets2014posterior} are fulfilled. Secondly, as described in Section 3.4 of \cite{fung2019moe1}, the denseness theorems do not provide a convergence rate, so there is no control on the mixture components $g$ to approximate any generalized mixed effects distributions at a desired level of accuracy. To establish the rate results, one needs to impose further assumptions on the target distribution $\tilde{H}(\bm{y}|\bm{x})$ in Equation (\ref{eq:gen_joint_cross2}) and the MMoE distribution $\tilde{F}(\bm{y};\bm{x})$ in Equation (\ref{eq:moe_joint_cross}). We leave these technical establishments as a future research direction.

\section{Acknowledgements}
This research did not receive any specific grant from funding agencies in the public, commercial, or not-for-profit sectors.

\bibliographystyle{abbrvnat}
\bibliography{references} 

\begin{thebibliography}{44}
\providecommand{\natexlab}[1]{#1}
\providecommand{\url}[1]{\texttt{#1}}
\expandafter\ifx\csname urlstyle\endcsname\relax
  \providecommand{\doi}[1]{doi: #1}\else
  \providecommand{\doi}{doi: \begingroup \urlstyle{rm}\Url}\fi

\bibitem[Aitkin and Longford(1986)]{aitkin1986statistical}
M.~Aitkin and N.~Longford.
\newblock Statistical modelling issues in school effectiveness studies.
\newblock \emph{Journal of the Royal Statistical Society: Series A (General)},
  149\penalty0 (1):\penalty0 1--26, 1986.

\bibitem[Bakker and Heskes(2003)]{bakker2003task}
B.~Bakker and T.~Heskes.
\newblock Task clustering and gating for bayesian multitask learning.
\newblock \emph{Journal of Machine Learning Research}, 4\penalty0
  (May):\penalty0 83--99, 2003.

\bibitem[Boucher and Denuit(2006)]{boucher2006fixed}
J.-P. Boucher and M.~Denuit.
\newblock Fixed versus random effects in {P}oisson regression models for claim
  counts: A case study with motor insurance.
\newblock \emph{ASTIN Bulletin: The Journal of the IAA}, 36\penalty0
  (1):\penalty0 285--301, 2006.

\bibitem[Breuer and Baum(2005)]{breuer2005introduction}
L.~Breuer and D.~Baum.
\newblock \emph{An introduction to queueing theory - and matrix-analytic
  methods}.
\newblock Springer Science \& Business Media, 2005.

\bibitem[Chamroukhi(2016)]{chamroukhi2016robust}
F.~Chamroukhi.
\newblock Robust mixture of experts modeling using the t distribution.
\newblock \emph{Neural Networks}, 79:\penalty0 20--36, 2016.

\bibitem[Chamroukhi(2017)]{chamroukhi2017skew}
F.~Chamroukhi.
\newblock Skew t mixture of experts.
\newblock \emph{Neurocomputing}, 266:\penalty0 390--408, 2017.

\bibitem[Davidian and Gallant(1993)]{davidian1993nonlinear}
M.~Davidian and A.~R. Gallant.
\newblock The nonlinear mixed effects model with a smooth random effects
  density.
\newblock \emph{Biometrika}, 80\penalty0 (3):\penalty0 475--488, 1993.

\bibitem[Frees and Kim(2006)]{frees2006multilevel}
E.~W. Frees and J.-S. Kim.
\newblock Multilevel model prediction.
\newblock \emph{Psychometrika}, 71\penalty0 (1):\penalty0 79--104, 2006.

\bibitem[Fung et~al.(2019{\natexlab{a}})Fung, Badescu, and Lin]{fung2019moe1}
T.~C. Fung, A.~L. Badescu, and X.~S. Lin.
\newblock A class of mixture of experts models for general insurance:
  {T}heoretical developments.
\newblock \emph{Insurance: Mathematics and Economics}, 89:\penalty0 111--127,
  2019{\natexlab{a}}.

\bibitem[Fung et~al.(2019{\natexlab{b}})Fung, Badescu, and Lin]{fung2019moe2}
T.~C. Fung, A.~L. Badescu, and X.~S. Lin.
\newblock A class of mixture of experts models for general insurance:
  {A}pplication to correlated claim frequencies.
\newblock \emph{ASTIN Bulletin: The Journal of the IAA}, 49\penalty0
  (3):\penalty0 647--688, 2019{\natexlab{b}}.

\bibitem[Fung et~al.(2021)Fung, Badescu, and Lin]{fung2020moe3}
T.~C. Fung, A.~L. Badescu, and X.~S. Lin.
\newblock A new class of severity regression models with an application to
  {IBNR} prediction.
\newblock \emph{North American Actuarial Journal}, 25\penalty0 (2):\penalty0
  206–231, 2021.

\bibitem[Geweke and Keane(2007)]{geweke2007smoothly}
J.~Geweke and M.~Keane.
\newblock Smoothly mixing regressions.
\newblock \emph{Journal of Econometrics}, 138\penalty0 (1):\penalty0 252--290,
  2007.

\bibitem[Goldstein(1986)]{goldstein1986multilevel}
H.~Goldstein.
\newblock Multilevel mixed linear model analysis using iterative generalized
  least squares.
\newblock \emph{Biometrika}, 73\penalty0 (1):\penalty0 43--56, 1986.

\bibitem[Goldstein(2011)]{goldstein2011multilevel}
H.~Goldstein.
\newblock \emph{Multilevel statistical models}, volume 922.
\newblock John Wiley \& Sons, 2011.

\bibitem[Gregoire and Schabenberger(1996)]{gregoire1996non}
T.~G. Gregoire and O.~Schabenberger.
\newblock A non-linear mixed-effects model to predict cumulative bole volume of
  standing trees.
\newblock \emph{Journal of Applied Statistics}, 23\penalty0 (2-3):\penalty0
  257--272, 1996.

\bibitem[Grun and Leisch(2008)]{grun2008flexmix}
B.~Grun and F.~Leisch.
\newblock Flexmix version 2: Finite mixtures with concomitant variables and
  varying and constant parameters.
\newblock \emph{Journal of Statistical Software}, 28\penalty0 (4):\penalty0
  1--35, 2008.

\bibitem[Ingrassia et~al.(2012)Ingrassia, Minotti, and
  Vittadini]{ingrassia2012local}
S.~Ingrassia, S.~C. Minotti, and G.~Vittadini.
\newblock Local statistical modeling via a cluster-weighted approach with
  elliptical distributions.
\newblock \emph{Journal of classification}, 29\penalty0 (3):\penalty0 363--401,
  2012.

\bibitem[Jacobs et~al.(1991)Jacobs, Jordan, Nowlan, and
  Hinton]{jacobs1991adaptive}
R.~A. Jacobs, M.~I. Jordan, S.~J. Nowlan, and G.~E. Hinton.
\newblock Adaptive mixtures of local experts.
\newblock \emph{Neural computation}, 3\penalty0 (1):\penalty0 79--87, 1991.

\bibitem[Jiang and Tanner(1999{\natexlab{a}})]{jiang1999approximation}
W.~Jiang and M.~A. Tanner.
\newblock On the approximation rate of hierarchical mixtures-of-experts for
  generalized linear models.
\newblock \emph{Neural computation}, 11\penalty0 (5):\penalty0 1183--1198,
  1999{\natexlab{a}}.

\bibitem[Jiang and Tanner(1999{\natexlab{b}})]{jiang1999hierarchical}
W.~Jiang and M.~A. Tanner.
\newblock Hierarchical mixtures-of-experts for exponential family regression
  models: approximation and maximum likelihood estimation.
\newblock \emph{Annals of Statistics}, pages 987--1011, 1999{\natexlab{b}}.

\bibitem[Jordan and Jacobs(1992)]{jordan1992hierarchies}
M.~I. Jordan and R.~A. Jacobs.
\newblock Hierarchies of adaptive experts.
\newblock In \emph{Advances in neural information processing systems}, pages
  985--992, 1992.

\bibitem[Jordan and Jacobs(1994)]{jordan1994hierarchical}
M.~I. Jordan and R.~A. Jacobs.
\newblock Hierarchical mixtures of experts and the {EM} algorithm.
\newblock \emph{Neural computation}, 6\penalty0 (2):\penalty0 181--214, 1994.

\bibitem[Masoudnia and Ebrahimpour(2014)]{masoudnia2014mixture}
S.~Masoudnia and R.~Ebrahimpour.
\newblock Mixture of experts: a literature survey.
\newblock \emph{Artificial Intelligence Review}, 42\penalty0 (2):\penalty0
  275--293, 2014.

\bibitem[McGilchrist(1994)]{mcgilchrist1994estimation}
C.~McGilchrist.
\newblock Estimation in generalized mixed models.
\newblock \emph{Journal of the Royal Statistical Society: Series B
  (Methodological)}, 56\penalty0 (1):\penalty0 61--69, 1994.

\bibitem[McLachlan and Peel(2000)]{mclachlan2000finite}
G.~J. McLachlan and D.~Peel.
\newblock \emph{Finite mixture models}.
\newblock John Wiley \& Sons, 2000.

\bibitem[Mendes and Jiang(2012)]{mendes2012convergence}
E.~F. Mendes and W.~Jiang.
\newblock On convergence rates of mixtures of polynomial experts.
\newblock \emph{Neural computation}, 24\penalty0 (11):\penalty0 3025--3051,
  2012.

\bibitem[Molenberghs et~al.(2010)Molenberghs, Verbeke, Dem{\'e}trio, and
  Vieira]{molenberghs2010family}
G.~Molenberghs, G.~Verbeke, C.~G. Dem{\'e}trio, and A.~M. Vieira.
\newblock A family of generalized linear models for repeated measures with
  normal and conjugate random effects.
\newblock \emph{Statistical science}, 25\penalty0 (3):\penalty0 325--347, 2010.

\bibitem[Ng and McLachlan(2007)]{ng2007extension}
S.-K. Ng and G.~J. McLachlan.
\newblock Extension of mixture-of-experts networks for binary classification of
  hierarchical data.
\newblock \emph{Artificial Intelligence in Medicine}, 41\penalty0 (1):\penalty0
  57--67, 2007.

\bibitem[Ng and McLachlan(2014)]{ng2014mixture}
S.-K. Ng and G.~J. McLachlan.
\newblock Mixture models for clustering multilevel growth trajectories.
\newblock \emph{Computational Statistics \& Data Analysis}, 71:\penalty0
  43--51, 2014.

\bibitem[Ng et~al.(2004)Ng, McLachlan, Yau, and Lee]{ng2004modelling}
S.-K. Ng, G.~McLachlan, K.~K. Yau, and A.~H. Lee.
\newblock Modelling the distribution of ischaemic stroke-specific survival time
  using an em-based mixture approach with random effects adjustment.
\newblock \emph{Statistics in Medicine}, 23\penalty0 (17):\penalty0 2729--2744,
  2004.

\bibitem[Ng et~al.(2006)Ng, McLachlan, Wang, Ben-Tovim~Jones, and
  Ng]{ng2006mixture}
S.-K. Ng, G.~J. McLachlan, K.~Wang, L.~Ben-Tovim~Jones, and S.-W. Ng.
\newblock A mixture model with random-effects components for clustering
  correlated gene-expression profiles.
\newblock \emph{Bioinformatics}, 22\penalty0 (14):\penalty0 1745--1752, 2006.

\bibitem[Nguyen and Chamroukhi(2018)]{nguyen2018practical}
H.~D. Nguyen and F.~Chamroukhi.
\newblock Practical and theoretical aspects of mixture-of-experts modeling: An
  overview.
\newblock \emph{Wiley Interdisciplinary Reviews: Data Mining and Knowledge
  Discovery}, 8\penalty0 (4):\penalty0 e1246, 2018.

\bibitem[Nguyen and McLachlan(2016)]{nguyen2016laplace}
H.~D. Nguyen and G.~J. McLachlan.
\newblock Laplace mixture of linear experts.
\newblock \emph{Computational Statistics \& Data Analysis}, 93:\penalty0
  177--191, 2016.

\bibitem[Nguyen et~al.(2016)Nguyen, Lloyd-Jones, and
  McLachlan]{nguyen2016universal}
H.~D. Nguyen, L.~R. Lloyd-Jones, and G.~J. McLachlan.
\newblock A universal approximation theorem for mixture-of-experts models.
\newblock \emph{Neural computation}, 28\penalty0 (12):\penalty0 2585--2593,
  2016.

\bibitem[Nguyen et~al.(2019)Nguyen, Chamroukhi, and
  Forbes]{nguyen2019approximation}
H.~D. Nguyen, F.~Chamroukhi, and F.~Forbes.
\newblock Approximation results regarding the multiple-output gaussian gated
  mixture of linear experts model.
\newblock \emph{Neurocomputing}, 366:\penalty0 208--214, 2019.

\bibitem[Nguyen et~al.(2021)Nguyen, Nguyen, Chamroukhi, and
  McLachlan]{nguyen2021approximations}
H.~D. Nguyen, T.~Nguyen, F.~Chamroukhi, and G.~J. McLachlan.
\newblock Approximations of conditional probability density functions in
  lebesgue spaces via mixture of experts models.
\newblock \emph{Journal of Statistical Distributions and Applications},
  8\penalty0 (1):\penalty0 1--15, 2021.

\bibitem[Norets and Pelenis(2014)]{norets2014posterior}
A.~Norets and J.~Pelenis.
\newblock Posterior consistency in conditional density estimation by covariate
  dependent mixtures.
\newblock \emph{Econometric Theory}, 30\penalty0 (3):\penalty0 606--646, 2014.

\bibitem[Norets et~al.(2010)]{norets2010approximation}
A.~Norets et~al.
\newblock Approximation of conditional densities by smooth mixtures of
  regressions.
\newblock \emph{The Annals of statistics}, 38\penalty0 (3):\penalty0
  1733--1766, 2010.

\bibitem[Tijms(1994)]{tijms1994stochastic}
H.~Tijms.
\newblock \emph{Stochastic Models: An Algorithm Approach}.
\newblock John Wiley, 1994.

\bibitem[Tseung et~al.(2022)Tseung, Fung, Chan, Badescu, and
  Lin]{Tseung2022MixedLRMoEApplication}
S.~C. Tseung, T.~C. Fung, I.~W. Chan, A.~L. Badescu, and X.~S. Lin.
\newblock A posteriori risk classification and ratemaking with random effects
  in the mixture-of-experts model.
\newblock Working paper, 2022.

\bibitem[Xu et~al.(1995)Xu, Jordan, and Hinton]{xu1995alternative}
L.~Xu, M.~I. Jordan, and G.~E. Hinton.
\newblock An alternative model for mixtures of experts.
\newblock \emph{Advances in neural information processing systems}, pages
  633--640, 1995.

\bibitem[Yau et~al.(2003)Yau, Lee, and Ng]{yau2003finite}
K.~K. Yau, A.~H. Lee, and A.~S. Ng.
\newblock Finite mixture regression model with random effects: application to
  neonatal hospital length of stay.
\newblock \emph{Computational statistics \& data analysis}, 41\penalty0
  (3-4):\penalty0 359--366, 2003.

\bibitem[Yuksel et~al.(2012)Yuksel, Wilson, and Gader]{yuksel2012twenty}
S.~E. Yuksel, J.~N. Wilson, and P.~D. Gader.
\newblock Twenty years of mixture of experts.
\newblock \emph{IEEE transactions on neural networks and learning systems},
  23\penalty0 (8):\penalty0 1177--1193, 2012.

\bibitem[Zeevi et~al.(1998)Zeevi, Meir, and Maiorov]{zeevi1998error}
A.~J. Zeevi, R.~Meir, and V.~Maiorov.
\newblock Error bounds for functional approximation and estimation using
  mixtures of experts.
\newblock \emph{IEEE Transactions on Information Theory}, 44\penalty0
  (3):\penalty0 1010--1025, 1998.

\end{thebibliography}

\begin{appendices}
\section{Proof of Theorem \ref{thm:denseness_cross}} \label{sec:apx1}
We first introduce the following functions, where the detailed notations would be explained later in the proof.
\begin{equation} \label{eq:apx1:1}
\tilde{H}(\bm{y}|\bm{x})=\int_{\tilde{\bm{\Theta}}}\left[\prod_{i=1}^{N}H(\bm{y}_i|\bm{x}_i,\bm{\theta}_i)\right]dG(\tilde{\bm{\theta}}),
\end{equation}
\begin{equation} \label{eq:apx1:2}
\tilde{H}^{*}(\bm{y}|\bm{x})=\sum_{\tilde{\bm{d}}\in\tilde{\mathcal{D}}}\prod_{i=1}^{N}H(\bm{y}_i|\bm{x}_i,\bm{\theta}^{*}_{\bm{d}^{(c(i))}})G(\tilde{\bm{\Theta}}_{\tilde{\bm{d}}}),
\end{equation}
\begin{equation} \label{eq:apx1:3}
\tilde{F}^{**(u)}(\bm{y}|\bm{x})=\int\prod_{i=1}^{N}F^{**(u)}(\bm{y}_i|\bm{x}_i,\bm{w}_i)d\Phi (\bm{w})
\end{equation}
with $F^{**(u)}(\bm{y}_i|\bm{x}_i,\bm{w}_i)=\sum_{\bm{d}\in\mathcal{D}^{+}}\xi_{\bm{d}}^{(u)}(\bm{w}_i)H(\bm{y}_i|\bm{x}_i,\bm{\theta}^{*}_{\bm{d}})$,
\begin{equation} \label{eq:apx1:4}
\tilde{F}^{*(M,Q,t,u)}(\bm{y}|\bm{x})=\int\prod_{i=1}^{N}F^{*(M,Q,t,u)}(\bm{y}_i|\bm{x}_i,\bm{w}_i)d\Phi (\bm{w})
\end{equation}
with $F^{*(M,Q,t,u)}(\bm{y}_i|\bm{x}_i,\bm{w}_i)=\sum_{\bm{m}\in\mathcal{M}}\sum_{\bm{q}\in\mathcal{Q}}\sum_{\bm{d}\in\mathcal{D}^{+}}\pi_j^{(t,u)}(\bm{x}_i,\bm{w}_i;\tilde{\bm{\alpha}}_{\bm{q}}^Q,\tilde{\bm{\beta}}_{\bm{d}})1\{\bm{y}_i\geq\lfloor\bm{y}\rfloor_{\bm{m}}^{M}\}$,
\begin{equation} \label{eq:apx1:5}
\tilde{F}^{(M,Q,t,u,v)}(\bm{y}|\bm{x})=\int\prod_{i=1}^{N}F^{(M,Q,t,u,v)}(\bm{y}_i|\bm{x}_i,\bm{w}_i)d\Phi (\bm{w})
\end{equation}
with $F^{(M,Q,t,u,v)}(\bm{y}_i|\bm{x}_i,\bm{w}_i)=\sum_{\bm{m}\in\mathcal{M}}\sum_{\bm{q}\in\mathcal{Q}}\sum_{\bm{d}\in\mathcal{D}^{+}}\pi_j^{(t,u)}(\bm{x}_i,\bm{w}_i;\tilde{\bm{\alpha}}_{\bm{q}}^Q,\tilde{\bm{\beta}}_{\bm{d}})F_0(\bm{y}_i;\bm{\psi}_{\bm{m}}^{M(v)})$. Note that Equation (\ref{eq:apx1:1}) is in the form of the generalized mixed effects models under Equation (\ref{eq:gen_joint_cross2}), while Equation (\ref{eq:apx1:5}) is in the MMoE framework under Equation (\ref{eq:moe_joint_cross}). The main proof idea is to bound the approximation errors between any two consecutive functions from Equations (\ref{eq:apx1:1}) to (\ref{eq:apx1:5}):

\subsection{Step 1: Approximating Equation (\ref{eq:apx1:1}) by Equation (\ref{eq:apx1:2})} \label{sec:apx1:step1}

As the metric space $(\bm{\Theta}_l,d_{\bm{\Theta}_l})$ is complete separable (Assumption \ref{asm:metric}), the probability measure $G_l$ on $\bm{\theta}_l^{(s)}$ is tight, i.e. $\forall \epsilon_1>0$, $\exists \bar{\bm{\Theta}}_l\subseteq \bm{\Theta}_l$ compact such that $G_l(\bar{\bm{\Theta}}_l):=\mathbb{P}(\bm{\theta}_l^{(s)}\in\bar{\bm{\Theta}}_l)\geq 1-\epsilon_1$. Since $\bar{\bm{\Theta}}_l$ is compact, for any $\delta_1>0$ we can find subspaces $\{\bm{\Theta}_{l,d_l}\}_{d_l=1,\ldots,D_l}$ ($d_l$ is called the subspace index) and points $\{\bm{\theta}^{*}_{l,d_l}\}_{d_l=1,\ldots,D_l}$ such that $\cup_{d_l=1,\ldots,D_l}\bm{\Theta}_{l,d_l}=\bar{\bm{\Theta}}_{l}$, $\bm{\theta}^{*}_{l,d_l}\in\bm{\Theta}_{l,d_l}$ for every $d_l=1,\ldots,D_l$ and $\bm{\Theta}_{l,d_l}$ is covered by the ball with radius $\delta_1/L$ centered at $\bm{\theta}^{*}_{l,d_l}$. Due to the uniform continuity of $H$ w.r.t. $\bm{\theta}_i$ on $\bar{\bm{\Theta}}_l$ (Assumption \ref{asm:continuity} implies uniform convergence in compact space), for any $\epsilon_2>0$ we can choose sufficiently small $\delta_1$ with the aforementioned subspace partitioning such that $|H(\bm{y}_i|\bm{x}_i,\bm{\theta}_i)-H(\bm{y}_i|\bm{x}_i,\bm{\theta}^{*}_{\bm{d}^{(c(i))}})|\leq\epsilon_2$ if $\bm{\theta}_{il}\in\bm{\Theta}_{l,d_l^{(c(i))}}$ for every $l=1,\ldots,L$, where $\bm{\theta}^{*}_{\bm{d}^{c(i)}}=\{\bm{\theta}^{*}_{1,d_l^{(c(i))}}\}_{l=1,\ldots,L}$ and $\bm{d}^{(c(i))}=\{d_l^{(c(i))}\}_{l=1,\ldots,L}$ with $d_l^{(c(i))}\in\{1,\ldots,D_l\}$. Note here the superscript $(c(i))$ of $d_l^{(c(i))}$ represents the subspace index $d_l$ corresponding to the $i^{\text{th}}$ observation.


Define $\tilde{H}^{*}(\bm{y}|\bm{x})$ in the form of Equation (\ref{eq:apx1:2}), where $\tilde{\mathcal{D}}=\prod_{l=1}^{L}\prod_{s=1}^{S_l}\mathcal{D}_l^{(s)}$ with $\mathcal{D}_l^{(s)}=\{1,\ldots,D_l\}$, $\tilde{\bm{d}}=\{d_l^{(s)}\}_{l=1,\ldots,L;s=1,\ldots,S_l}$ and $\tilde{\bm{\Theta}}_{\tilde{\bm{d}}}=\prod_{l=1}^{L}\prod_{s=1}^{S_l}\bm{\Theta}_{l,d_l^{(s)}}$. We first introduce the following technical lemma which can be easily proved by induction:

\begin{lemma} \label{lm:product}
Given $0\leq a_i,b_i\leq 1$ and $|a_i-b_i|\leq\epsilon$ for all $i=1,\ldots,N$, then $|\prod_{i=1}^{N}a_i-\prod_{i=1}^{N}b_i|\leq N\epsilon$.
\end{lemma}

Since $\mathbb{P}(\tilde{\bm{\theta}}\notin\cup_{\tilde{\bm{d}}\in\tilde{\bm{D}}}\tilde{\bm{\Theta}}_{\tilde{\bm{d}}})\leq\sum_{l=1,\ldots,L;s=1,\ldots,S_l}\mathbb{P}(\bm{\theta}_l^{(s)}\notin\bar{\bm{\Theta}}_l)\leq NL\epsilon_1$, we can now rewrite $\tilde{H}(\bm{y}|\bm{x})$ as
\begin{equation} \label{eq:apx1:step1:rewrite1}
\tilde{H}(\bm{y}|\bm{x})=\sum_{\tilde{\bm{d}}\in\tilde{\mathcal{D}}}\int_{\tilde{\bm{\Theta}}_{\tilde{\bm{d}}}}\prod_{i=1}^{N}H(\bm{y}_i|\bm{x}_i,\bm{\theta}_i)dG(\tilde{\bm{\theta}})+NL\mathcal{O}_1(\epsilon_1),
\end{equation}
and $\tilde{H}^{*}(\bm{y}|\bm{x})$ as
\begin{equation} \label{eq:apx1:step1:rewrite2}
\tilde{H}^{*}(\bm{y}|\bm{x})=\sum_{\tilde{\bm{d}}\in\tilde{\mathcal{D}}}\int_{\tilde{\bm{\Theta}}_{\tilde{\bm{d}}}}\prod_{i=1}^{N}H(\bm{y}_i|\bm{x}_i,\bm{\theta}^{*}_{\bm{d}^{(c(i))}})dG(\tilde{\bm{\theta}}),
\end{equation}
where $0\leq\mathcal{O}_1(\epsilon_1)\leq\epsilon_1$, and we have the following approximation error bound between $\tilde{H}(\bm{y}|\bm{x})$ and $\tilde{H}^{*}(\bm{y}|\bm{x})$:
\begin{align} \label{eq:apx1:step1:bound_final}
|\tilde{H}^{*}(\bm{y}|\bm{x})-\tilde{H}(\bm{y}|\bm{x})|
&\leq\sum_{\tilde{\bm{d}}\in\tilde{\mathcal{D}}}\int_{\tilde{\bm{\Theta}}_{\tilde{\bm{d}}}}
\left|\prod_{i=1}^{N}H(\bm{y}_i|\bm{x}_i,\bm{\theta}^{*}_{\bm{d}^{(c(i))}})-\prod_{i=1}^{N}H(\bm{y}_i|\bm{x}_i,\bm{\theta}_i)\right|
dG(\tilde{\bm{\theta}})+NL\epsilon_1 \nonumber\\
&\leq N\epsilon_2+NL\epsilon_1,
\end{align}
where the second inequality is resulted from Lemma \ref{lm:product}.

\subsection{Step 2: Approximating Equation (\ref{eq:apx1:2}) by Equation (\ref{eq:apx1:3})} \label{sec:apx1:step2}

Partition the space of $w_{il}$ (i.e. $\mathbb{R}$) into $D_l+2$ adjacent half open half closed intervals $\mathcal{W}_{l,d_l}=(w^{*}_{l,d_l-1},w^{*}_{l,d_l}]$ (for $d_l=1,\ldots,D_l$), $\mathcal{W}_{l,0}=(-\infty,w^{*}_{l,0}]$ and $\mathcal{W}_{l,D_l+1}=(w^{*}_{l,D_l},\infty)$ such that $\Phi_l(\mathcal{W}_{l,d_l})=G_l(\bm{\Theta}_{l,d_l})$ for every $d_l=1,\ldots,D_l$. Note that such a partitioning always exists as $\Phi_l$ is a continuous distribution. Also denote the domain of (the $L$-vector) $\bm{d}=(d_1,\ldots,d_L)$ as $\mathcal{D}=\prod_{l=1}^{L}\{1,\ldots,D_l\}$ and the corresponding extended domain $\mathcal{D}^{+}=\prod_{l=1}^{L}\{0,1,\ldots,D_l+1\}$.


Denote $\tilde{F}^{**(u)}(\bm{y}|\bm{x})$ as the form of Equation (\ref{eq:apx1:3}) with $\xi^{(u)}_{\bm{d}}(\bm{w}_i)$ given by
\begin{equation} \label{eq:apx1:step2:xi}
\xi^{(u)}_{\bm{d}}(\bm{w}_i)=\exp\{u(\tilde{\beta}_{\bm{d},0}+\tilde{\bm{\beta}}_{\bm{d}}^T\bm{w}_i)\}/\sum_{\bm{d}'\in\mathcal{D}^{+}}\exp\{u(\tilde{\beta}_{\bm{d}',0}+\tilde{\bm{\beta}}_{\bm{d}'}^T\bm{w}_i)\}.
\end{equation}

We have the following technical lemma to help us choose suitable parameters $\{(\tilde{\beta}_{\bm{d},0},\tilde{\bm{\beta}}_{\bm{d}})\}_{{\bm{d}}\in\mathcal{D}^{+}}$ of $\xi^{(u)}_{\bm{d}}(\bm{w}_i)$:

\begin{lemma} \label{lm:xi}
There exists parameters $\{(\tilde{\beta}_{\bm{d},0},\tilde{\bm{\beta}}_{\bm{d}})\}_{{\bm{d}}\in\mathcal{D}^{+}}$ of $\xi^{(u)}_{\bm{d}}(\bm{w}_i)$ such that $\xi^{(u)}_{\bm{d}}(\bm{w}_i)\xrightarrow{u\rightarrow\infty}\prod_{l=1}^{L}1^{*}_{w_{il}}(\mathcal{W}_{l,d_l})$ for every $\bm{d}\in\mathcal{D}^{+}$, where $1^{*}_w(\mathcal{W})$ is an indicator which equals to 1 if $w$ falls inside the interior of $\mathcal{W}$, equals to a number between 0 and 1 if $W$ is on the boundary of $\mathcal{W}$, and equals to 0 otherwise.
\end{lemma}

\begin{proof}
With a slight (notational) variant of Lemma 3.1 of \cite{fung2019moe1}, it can be easily shown that there exists parameters $\{(\tilde{\beta}_{\bm{d},0},\tilde{\bm{\beta}}_{\bm{d}})\}_{{\bm{d}}\in\mathcal{D}^{+}}$ such that $\tilde{\beta}_{\bm{d},0}+\tilde{\bm{\beta}}_{\bm{d}}^T\bm{w}_i>\max_{\bm{d}'\neq\bm{d};\bm{d}'\in\mathcal{D}^{+}}\tilde{\beta}_{\bm{d}',0}+\tilde{\bm{\beta}}_{\bm{d}'}^T\bm{w}_i$ if and only if $w_{il}\in\mathcal{W}_{l,d_l}^{*}$ for every $\bm{d}\in\mathcal{D}^{+}$, where $\mathcal{W}_{l,d_l}^{*}$ is the interior of $\mathcal{W}_{l,d_l}$. Under a slight notational variant of Lemma 3.2 of \cite{fung2019moe1}, we have $\xi^{(u)}_{\bm{d}}(\bm{w}_i)\xrightarrow{u\rightarrow\infty}1\{\tilde{\beta}_{\bm{d},0}+\tilde{\bm{\beta}}_{\bm{d}}^T\bm{w}_i>\max_{\bm{d}'\neq\bm{d};\bm{d}'\in\mathcal{D}^{+}}\tilde{\beta}_{\bm{d}',0}+\tilde{\bm{\beta}}_{\bm{d}'}^T\bm{w}_i\}+\mathcal{O}_1(1)\times 1\{\tilde{\beta}_{\bm{d},0}+\tilde{\bm{\beta}}_{\bm{d}}^T\bm{w}_i=\max_{\bm{d}'\neq\bm{d};\bm{d}'\in\mathcal{D}^{+}}\tilde{\beta}_{\bm{d}',0}+\tilde{\bm{\beta}}_{\bm{d}'}^T\bm{w}_i\}=\prod_{l=1}^{L}1^{*}_{w_{il}}(\mathcal{W}_{l,d_l})$, where $0\leq\mathcal{O}_1(1)\leq 1$.
\end{proof}

Choosing the parameters indicated by the above lemma, we have the following approximation result between $\tilde{H}^{*}(\bm{y}|\bm{x})$ and $\tilde{F}^{**(u)}(\bm{y}|\bm{x})$:

\allowdisplaybreaks
\begin{align} \label{eq:apx1:step2:bound_final}
\tilde{F}^{**(u)}(\bm{y}|\bm{x})
&\xrightarrow{u\rightarrow\infty}\int\prod_{i=1}^{N}\sum_{\bm{d}\in\mathcal{D}^{+}}\left(\prod_{l=1}^{L}1^{*}_{w_{il}}(\mathcal{W}_{l,d_l})\right)H(\bm{y}_i|\bm{x}_i,\bm{\theta}^{*}_{\bm{d}})d\Phi(\bm{w})\nonumber\\
&=\int_{\bar{\mathcal{W}}}\prod_{i=1}^{N}\sum_{\bm{d}\in\mathcal{D}^{+}}\left(\prod_{l=1}^{L}1^{*}_{w_{il}}(\mathcal{W}_{l,d_l})\right)H(\bm{y}_i|\bm{x}_i,\bm{\theta}^{*}_{\bm{d}})d\Phi(\bm{w})+\mathcal{O}_1(\epsilon_1)\nonumber\\
&=\sum_{\tilde{\bm{d}}\in\tilde{\mathcal{D}}}\int_{\mathcal{W}_{\tilde{\bm{d}}}}\prod_{i=1}^{N}\sum_{\bm{d}\in\mathcal{D}}\left(\prod_{l=1}^{L}1^{*}_{w_{il}}(\mathcal{W}_{l,d_l})\right)H(\bm{y}_i|\bm{x}_i,\bm{\theta}^{*}_{\bm{d}})d\Phi(\bm{w})+\mathcal{O}_1(\epsilon_1)\nonumber\\
&=\sum_{\tilde{\bm{d}}\in\tilde{\mathcal{D}}}\int_{\mathcal{W}_{\tilde{\bm{d}}}}\prod_{i=1}^{N}H(\bm{y}_i|\bm{x}_i,\bm{\theta}^{*}_{\bm{d}^{(c(i))}})d\Phi(\bm{w})+\mathcal{O}_1(\epsilon_1)\nonumber\\
&=\sum_{\tilde{\bm{d}}\in\tilde{\mathcal{D}}}\prod_{i=1}^{N}H(\bm{y}_i|\bm{x}_i,\bm{\theta}^{*}_{\bm{d}^{(c(i))}})G(\tilde{\bm{\Theta}}_{\tilde{\bm{d}}})+\mathcal{O}_1(\epsilon_1)\nonumber\\
&=\tilde{H}^{*}(\bm{y}|\bm{x})+\mathcal{O}_1(\epsilon_1),
\end{align}
where $\bar{\mathcal{W}}_l^{(s)}=\cup_{d_l=1}^{D_l}\mathcal{W}_{l,d_l}$, $\bar{\mathcal{W}}=\prod_{l=1}^{L}\prod_{s=1}^{S_l}\bar{\mathcal{W}}_l^{(s)}$ and $\mathcal{W}_{\tilde{\bm{d}}}=\prod_{l=1}^{L}\prod_{s=1}^{S_l}\bar{\mathcal{W}}_{l,d_l^{(s)}}$. The convergence is resulted from the Dominated Convergence Theorem (DCT), which is obviously uniform on $(\bm{y},\bm{x})$ as $\xi^{(u)}_{\bm{d}}(\bm{w}_i)$ (the only term in $\tilde{F}^{**(u)}(\bm{y}|\bm{x})$ which depends on $u$) does not depend on $(\bm{y},\bm{x})$ and $H(\bm{y}_i|\bm{x}_i,\bm{\theta}^{*}_{\bm{d}})$ is bounded above at $1$. The third last equality holds as the events of the indicator functions only (jointly) hold if $\bm{d}=\bm{d}^{(c(i))}$ when $\bm{w}\in\mathcal{W}_{\tilde{\bm{d}}}$. The second last equality holds as the integrand is constant on $\bm{w}\in\mathcal{W}_{\tilde{\bm{d}}}$ and $\Phi(\mathcal{W}_{\tilde{\bm{d}}})=\prod_{l=1}^{L}\prod_{s=1}^{S_l}\Phi_l(\mathcal{W}_{l,d_l^{(s)}})=\prod_{l=1}^{L}\prod_{s=1}^{S_l}G_l(\bm{\Theta}_{l,d_l^{(s)}})=G(\tilde{\bm{\Theta}}_{\tilde{\bm{d}}})$.

\subsection{Step 3: Approximating Equation (\ref{eq:apx1:3}) by Equation (\ref{eq:apx1:4})} \label{sec:apx1:step3}

Partition the space of $y_{ik}\in\mathbb{R}$ into adjacent half open half closed intervals $\{\mathcal{Y}_{k,m}^M=((m-1/2)h_M^y,(m+1/2)h_M^y]\}_{m=-M,\ldots,M-1,M}$. Define $\mathcal{Y}^M:=\prod_{k=1}^K\cup_{m=-M}^{M}\mathcal{Y}_{k,m}^M:=\prod_{k=1}^K\mathcal{Y}_{k}^M=(-(M+1/2)h_M^y,(M+1/2)h_M^y]^K$, $\mathcal{Y}_{k,-(M+1)}^M=(-\infty,-(M+1/2)h_M^y]$ and $\mathcal{Y}_{k,M+1}^M=((M+1/2)h_M^y,\infty)$. Choose $h_M^y$ such that $h_M^y\downarrow 0$ and $Mh_M^y\uparrow\infty$ as $M\rightarrow\infty$. Also denote $\mathcal{Y}_{\bm{m}}^{M}=\mathcal{Y}_{1,m_1}^M\times\cdots\times\mathcal{Y}_{K,m_K}^M$ as a hypercube with $\bm{m}:=(m_1,\ldots,m_K)\in\mathcal{M}:=\{-(M+1),\ldots,(M+1)\}^K$.


Similarly, partition the space of $x_{ip}\in\mathbb{R}$ into adjacent half open half closed intervals $\{\mathcal{X}_{p,q}^Q=((q-1/2)h_Q^x,(q+1/2)h_Q^x]\}_{q=-Q,\ldots,Q-1,Q}$. Define $\mathcal{X}^Q:=\prod_{p=1}^P\cup_{q=-Q}^{Q}\mathcal{X}_{p,q}^Q:=\prod_{p=1}^P\mathcal{X}_{p}^Q=(-(Q+1/2)h_Q^x,(Q+1/2)h_Q^x]^P$, $\mathcal{X}_{p,-(Q+1)}^Q=(-\infty,-(Q+1/2)h_Q^x]$ and $\mathcal{X}_{p,Q+1}^Q=((Q+1/2)h_Q^x,\infty)$. Choose $h_Q^x$ such that $h_Q^x\downarrow 0$ and $Qh_Q^x\uparrow\infty$ as $Q\rightarrow\infty$. Also denote $\mathcal{X}_{\bm{q}}^{Q}=\mathcal{Y}_{1,q_1}^Q\times\cdots\times\mathcal{X}_{P,q_P}^Q$ as a hypercube with $\bm{q}:=(q_1,\ldots,q_P)\in\mathcal{Q}:=\{-(Q+1),\ldots,(Q+1)\}^P$.

Denote $\tilde{F}^{*(M,Q,t,u)}(\bm{y}|\bm{x})$ as the form of Equation (\ref{eq:apx1:4}), where $\lfloor\bm{y}\rfloor_{\bm{m}}^{M}:=(\lfloor y\rfloor_{m_1}^{M},\ldots,\lfloor y\rfloor_{m_K}^{M})$ (inside the expression of $F^{*(M,Q,t,u)}(\bm{y}_i|\bm{x}_i,\bm{w}_i)$) represents the leftmost vertex of the hypercube $\mathcal{Y}_{\bm{m}}^{M}$. Also, the function $\pi_j^{(t,u)}(\bm{x}_i,\bm{w}_i;\tilde{\bm{\alpha}}_{\bm{q}}^Q,\tilde{\bm{\beta}}_{\bm{d}})$ is a logit linear gating function given by

\begin{align}
&\pi_j^{(t,u)}(\bm{x}_i,\bm{w}_i;\tilde{\bm{\alpha}}_{\bm{q}}^Q,\tilde{\bm{\beta}}_{\bm{d}})\nonumber\\
&=\frac{\exp\{\log H(\mathcal{Y}_{\bm{m}}^M|\bm{x}_{\bm{q}}^{*Q},\bm{\theta}_{\bm{d}}^{*})+t(\tilde{\alpha}_{\bm{q},0}^Q+\tilde{\bm{\alpha}}_{\bm{q}}^{QT}\bm{x}_i)+u(\tilde{\beta}_{\bm{d},0}+\tilde{\bm{\beta}}_{\bm{d}}^T\bm{w}_i)\}}{\sum_{\bm{m}'\in\mathcal{M}}\sum_{\bm{q}'\in\mathcal{Q}}\sum_{\bm{d}'\in\mathcal{D}^{+}}\exp\{\log H(\mathcal{Y}_{\bm{m}'}^M|\bm{x}_{\bm{q}'}^{*Q},\bm{\theta}_{\bm{d}'}^{*})+t(\tilde{\alpha}_{\bm{q}',0}^Q+\tilde{\bm{\alpha}}_{\bm{q}'}^{QT}\bm{x}_i)+u(\tilde{\beta}_{\bm{d}',0}+\tilde{\bm{\beta}}_{\bm{d}'}^T\bm{w}_i)\}}\\
&=\frac{\exp\{t(\tilde{\alpha}_{\bm{q},0}^Q+\tilde{\bm{\alpha}}_{\bm{q}}^{QT}\bm{x}_i)\}}{\sum_{\bm{q}'\in\mathcal{Q}}\exp\{t(\tilde{\alpha}_{\bm{q}',0}^Q+\tilde{\bm{\alpha}}_{\bm{q}'}^{QT}\bm{x}_i)\}}\frac{\exp\{u(\tilde{\beta}_{\bm{d},0}+\tilde{\bm{\beta}}_{\bm{d}}^T\bm{w}_i)\}}{\sum_{\bm{d}'\in\mathcal{D}^{+}}\exp\{u(\tilde{\beta}_{\bm{d}',0}+\tilde{\bm{\beta}}_{\bm{d}'}^T\bm{w}_i)\}}H(\mathcal{Y}_{\bm{m}}^M|\bm{x}_{\bm{q}}^{*Q},\bm{\theta}_{\bm{d}}^{*})\nonumber\\
&:=\gamma_{\bm{q}}^{(t)}(\bm{x}_i)\times\xi_{\bm{d}}^{(u)}(\bm{w}_i)\times H(\mathcal{Y}_{\bm{m}}^M|\bm{x}_{\bm{q}}^{*Q},\bm{\theta}_{\bm{d}}^{*}),
\end{align}
where $\bm{x}_{\bm{q}}^{*Q}=(x_{1,q_1}^{*Q},\ldots,x_{P,q_P}^{*Q})$ is the mid-point of the hypercube $\mathcal{X}_{\bm{q}}^Q$. Here, ``mid-points" for $\mathcal{X}_{p,-(Q+1)}$ and $\mathcal{X}_{p,(Q+1)}$ are respectively defined as $-(Q+1)h_Q^x$ and $(Q+1)h_Q^x$.


Further denote $R^Q(\bm{x}_i)=\{\bm{q}:\forall p ~\text{we have} ~|pj^Q(x_{ip})-x_{p,q_p}^{*Q}|\leq h_Q^x\}$ and $R^{'Q}(\bm{x}_i)=\{\bm{q}:\exists p~\text{such that}~|pj^Q(x_{ip})-x_{p,q_p}^{*Q}|> h_Q^x\}$, where $pj^Q(x_{ip})$ is the projection of $x_{ip}$ on the interval $\mathcal{X}_{p}^Q$ (i.e. $pj^Q(x_{ip})$ is the point in $\mathcal{X}_{p}^Q$ where $x_{ip}$ is closest to). To derive the approximation results, we first introduce the following two technical lemmas.

\begin{lemma} \label{lm:gamma}
Given any fixed $Q$ and $h_Q^x$, there exists parameters $\{\tilde{\alpha}_{\bm{q},0}^Q,\tilde{\bm{\alpha}}_{\bm{q}}^{QT}\}_{\bm{q}\in\mathcal{Q}}$, such that for any $\epsilon_3>0$, we have $\sum_{\bm{q}\in R^{'Q}(\bm{x}_i)}\gamma_{\bm{q}}^{(t)}(\bm{x}_i)\leq\epsilon_3$ for all $\bm{x}_i\in\mathcal{X}$ with sufficiently large $t$.
\end{lemma}

\begin{proof}
This follows directly from the proof of Theorem 3.2 of \cite{fung2019moe1}.
\end{proof}

\begin{lemma} \label{lm:tight}
The probability distributions $\{H(\bm{y}_i|\bm{x}_i,\bm{\theta}_i)\}_{\bm{x}_i\in\bar{\mathcal{X}};\bm{\theta}_i\in\bar{\bm{\Theta}}}$ are tight for any compact spaces $\bar{\mathcal{X}}$ and $\bar{\bm{\Theta}}$.
\end{lemma}

\begin{proof}
Divide the compact space $\bar{\mathcal{Z}}:=\bar{\bm{X}}\times\bar{\bm{\Theta}}$ into $D$ subspaces $\{\mathcal{Z}_d\}_{d=1,\ldots,D}$, where each subspace $\mathcal{Z}_d$ is small enough to be covered by a ball with radius $\delta$. Define $\bm{z}_d^*:=(\bm{x}_d^*,\bm{\theta}_d^*)$ as an arbitrary interior point of $\mathcal{Z}_d$ for $d=1,\ldots,D$. For each $d=1,\ldots,D$, we choose a response space $\tilde{\mathcal{Y}}_d\in\mathcal{Y}$ such that $H(\tilde{\mathcal{Y}}_d|\bm{x}_d^*,\bm{\theta}_d^*)\geq 1-\epsilon /2$. Select a compact space $\tilde{\mathcal{Y}}^*$ which covers all $\{\tilde{\mathcal{Y}}_d\}_{d=1,\ldots,D}$, and we have $H(\tilde{\mathcal{Y}}^*|\bm{x}_d^*,\bm{\theta}_d^*)\geq 1-\epsilon /2$ true for all $d=1,\ldots,D$. Uniform continuity of $H$ on any compact space implies that $|H(\tilde{\mathcal{Y}}^*|\bm{x}_d^*,\bm{\theta}_d^*)-H(\tilde{\mathcal{Y}}^*|\bm{x}_i,\bm{\theta}_i)|\leq\epsilon /2$ for sufficient small $\delta$ if $(\bm{x}_i,\bm{\theta}_i)\in\mathcal{Z}_d$. Overall, we have $H(\tilde{\mathcal{Y}}^*|\bm{x}_i,\bm{\theta}_i)\geq 1-\epsilon$, and hence the result follows.
\end{proof}

As a result, for any compact covariates space of $\bar{\mathcal{X}}$ and any $\epsilon_4>0$, we can use Lemma \ref{lm:tight} to select a rectangular output space $\bar{\mathcal{Y}}$ such that $H(\bar{\mathcal{Y}}|\bm{x}_i,\bm{\theta}_i)\geq 1-\epsilon_4$ for any $\bm{x}_i\in\bar{\mathcal{X}}$ and $\bm{\theta}_i\in\bar{\bm{\Theta}}$, where $\bar{\bm{\Theta}}=\prod_{l=1}^{L}\bar{\bm{\Theta}}_l$ and note that $\bar{\bm{\Theta}}_l$ is defined in Section \ref{sec:apx1:step1}. Then we have for all $\bm{x}_i,\bm{x}_{\bm{q}}^{*Q}\in\bar{\mathcal{X}}$ and $\bm{\theta}^{*}_{\bm{d}}\in\bar{\bm{\Theta}}$:
\begin{equation}
|H(\lceil\bm{y}\rceil_{\bm{m}}^M|\bm{x}_{\bm{q}}^{*Q},\bm{\theta}^{*}_{\bm{d}})-H(pj^M(\lceil\bm{y}\rceil_{\bm{m}}^M)|\bm{x}_q^{*Q},\bm{\theta}^{*}_{\bm{d}})|\leq \epsilon_4,
\end{equation}
and
\begin{equation}
|H(\bm{y}_i|\bm{x}_i,\bm{\theta}^{*}_{\bm{d}})-H(pj^M(\bm{y}_i)|\bm{x}_i,\bm{\theta}^{*}_{\bm{d}})|\leq \epsilon_4,
\end{equation}
where $pj^M(\bm{y}_i)$ is the projection of $\bm{y}_i$ on $\bar{\mathcal{Y}}$, and $\lceil\bm{y}\rceil_{\bm{m}}^M:=(\lceil y \rceil_{m_1}^M,\ldots,\lceil y \rceil_{m_K}^M)$ is the rightmost vertex of hypercube $\mathcal{Y}_{\bm{m}}^{M}$. Further, because of the uniform continuity of $H(\bm{y}_i|\bm{x}_i,\bm{\theta}^{*}_{\bm{d}})$ on $(\bm{x}_i,\bm{\theta}^{*}_{\bm{d}})$ within a compact support, for any $\epsilon_5>0$, we can choose sufficient large $Q$ (to make $h_Q^x$ small enough while $\mathcal{X}^Q$ covers $\bar{\mathcal{X}}$) and $M$ (to make $h_M^y$ small enough while $\mathcal{Y}^M$ covers $\bar{\mathcal{Y}}$) such that for any $\bm{y}_i\in\mathcal{Y}_{\bm{m}}^M$ and $\bm{q}\in R^Q(\bm{x}_i)$, we have:
\begin{equation}
|H(pj^M(\bm{y}_i)|\bm{x}_i,\bm{\theta}^{*}_{\bm{d}})-H(pj^M(\lceil\bm{y}\rceil_{\bm{m}}^M)|\bm{x}_{\bm{q}}^{*Q},\bm{\theta}^{*}_{\bm{d}})|\leq\epsilon_5.
\end{equation}

Summarizing the above three equations, we have:
\begin{equation} \label{eq:apx1:step3:bound}
|H(\bm{y}_i|\bm{x}_i,\bm{\theta}^{*}_{\bm{d}})-H(\lceil\bm{y}\rceil_{\bm{m}}^M|\bm{x}_{\bm{q}}^{*Q},\bm{\theta}^{*}_{\bm{d}})|\leq 2\epsilon_4+\epsilon_5.
\end{equation}

Note that $F^{*(M,Q,t,u)}(\bm{y}_i|\bm{x}_i,\bm{w}_i)$ can be re-written as
\begin{align} \label{eq:apx1:step3:rewrite}
F^{*(M,Q,t,u)}(\bm{y}_i|\bm{x}_i,\bm{w}_i)
&=\sum_{\bm{m}\in\mathcal{M}}\sum_{\bm{q}\in\mathcal{Q}}\sum_{\bm{d}\in\mathcal{D}^{+}}\gamma_{\bm{q}}^{(t)}(\bm{x}_i)\xi_{\bm{d}}^{(u)}(\bm{w}_i) H(\lceil\bm{y}\rceil_{\bm{m}}^M|\bm{x}_{\bm{q}}^{*Q},\bm{\theta}_{\bm{d}}^{*})1\{\bm{y}_i\in\mathcal{Y}_{\bm{m}}^M\}\nonumber\\
&=\sum_{\bm{m}\in\mathcal{M}}\sum_{\bm{q}\in R^Q(\bm{x}_i)}\sum_{\bm{d}\in\mathcal{D}^{+}}\gamma_{\bm{q}}^{(t)}(\bm{x}_i)\xi_{\bm{d}}^{(u)}(\bm{w}_i) H(\lceil\bm{y}\rceil_{\bm{m}}^M|\bm{x}_{\bm{q}}^{*Q},\bm{\theta}_{\bm{d}}^{*})1\{\bm{y}_i\in\mathcal{Y}_{\bm{m}}^M\}+\mathcal{O}_1(\epsilon_3),
\end{align}
where $0\leq\mathcal{O}_1(\epsilon_3)\leq \epsilon_3$. Note that the last equality is resulted from Lemma \ref{lm:gamma}. Then, the approximation result is given by:

\begin{align} \label{eq:apx1:step3:bound_final}
&|\tilde{F}^{*(M,Q,t,u)}(\bm{y}|\bm{x})-\tilde{F}^{**(u)}(\bm{y}|\bm{x})|\nonumber\\
&\leq \int\left|\prod_{i=1}^{N}F^{*(M,Q,t,u)}(\bm{y}_i|\bm{x}_i,\bm{w}_i)-\prod_{i=1}^{N}F^{**(u)}(\bm{y}_i|\bm{x}_i,\bm{w}_i)\right|d\Phi (\bm{w})\nonumber\\
&\leq N\int \left|F^{*(M,Q,t,u)}(\bm{y}_i|\bm{x}_i,\bm{w}_i)-F^{**(u)}(\bm{y}_i|\bm{x}_i,\bm{w}_i)\right| d\Phi (\bm{w})\nonumber\\
&\leq N\int \sum_{\bm{m}\in\mathcal{M}}\sum_{\bm{q}\in R^Q(\bm{x}_i)}\sum_{\bm{d}\in\mathcal{D}^{+}}\gamma_{\bm{q}}^{(t)}(\bm{x}_i)\xi_{\bm{d}}^{(u)}(\bm{w}_i)\left|H(\bm{y}_i|\bm{x}_i,\bm{\theta}^{*}_{\bm{d}})-H(\lceil\bm{y}\rceil_{\bm{m}}^M|\bm{x}_{\bm{q}}^{*Q},\bm{\theta}^{*}_{\bm{d}})\right| 1\{\bm{y}_i\in\mathcal{Y}_{\bm{m}}^M\} d\Phi (\bm{w})+2\epsilon_3\nonumber\\
&\leq N(2\epsilon_4+\epsilon_5+2\epsilon_3),
\end{align}
where the second inequality is resulted from Lemma \ref{lm:product}, the third and last inequalities are respectively resulted from Equations (\ref{eq:apx1:step3:rewrite}) and (\ref{eq:apx1:step3:bound}).

\subsection{Step 4: Approximating Equation (\ref{eq:apx1:4}) by Equation (\ref{eq:apx1:5})} \label{sec:apx1:step4}

In the final step, we denote $\tilde{F}^{(M,Q,t,u,v)}(\bm{y}|\bm{x})$ as the form of Equation (\ref{eq:apx1:5}), where $F_0(\bm{y}_i;\bm{\psi}_{\bm{m}}^{M(v)})$ in $F^{(M,Q,t,u,v)}(\bm{y}_i|\bm{x}_i,\bm{w}_i)$ is chosen in the way that $F_0(\bm{y}_i;\bm{\psi}_{\bm{m}}^{M(v)})\xrightarrow{\mathcal{D}}1\{\bm{y}_i\geq\lfloor\bm{y}\rfloor_{\bm{m}}^{M}\}$ as $v\rightarrow\infty$. Due to the distributional convergence as well as $\mathcal{M}$ is a finite set, for any $\epsilon_6>0$ we can find a sufficient large $v$ such that for every $\bm{m}\in\mathcal{M}$, we have:
\begin{equation} \label{eq:apx1:step4:approx}
|F_0(\bm{y}_i;\bm{\psi}_{\bm{m}}^{M(v)})-1\{\bm{y}_i\geq\lfloor\bm{y}\rfloor_{\bm{m}}^{M}\}|\leq \epsilon_6 + \sum_{k=1}^{K}1\{y_{ik}\in\mathcal{L}_k^{\delta^{*}}(\lfloor y\rfloor_{m_k}^{M})\},
\end{equation}
where $\delta^{*}$ is chosen to be $0<\delta^{*}<h_M^y/2$, and $\mathcal{L}_k^{\delta^{*}}(\lfloor y\rfloor_{m_k}^{M})$ represents non-overlapping intervals for $m_k=-(M+1),\ldots,(M+1)$ with
\begin{equation}
  \mathcal{L}_k^{\delta^{*}}(\lfloor y\rfloor_{m_k}^{M})=\left\{
  \begin{array}{@{}ll@{}}
    [\lfloor y\rfloor_{m_k}^{M}-\delta^{*},\lfloor y\rfloor_{m_k}^{M}+\delta^{*}], & \text{if}\ m_k>-(M+1) \\
    (-\infty,\lfloor y\rfloor_{m_k}^{M}-2\delta^{*}], & m_k=-(M+1)
  \end{array}\right.
\end{equation} 
Note that the rightmost term in Equation (\ref{eq:apx1:step4:approx}) is to control for the fact that the weak convergence of $F_0$ to the indicator is not uniform when $y_{ik}$ is close to $\lfloor y\rfloor_{m_k}^{M}$. Consider the bound
\begin{align}\label{apx1:step4:approx2}
&|F^{(M,Q,t,u,v)}(\bm{y}_i|\bm{x}_i,\bm{w}_i)-F^{*(M,Q,t,u)}(\bm{y}_i|\bm{x}_i,\bm{w}_i)|\nonumber\\
&\leq \sum_{\bm{m}\in\mathcal{M}}\sum_{\bm{q}\in\mathcal{Q}}\sum_{\bm{d}\in\mathcal{D}^{+}}\gamma_{\bm{q}}^{(t)}(\bm{x}_i)\xi_{\bm{d}}^{(u)}(\bm{w}_i) H(\mathcal{Y}_{\bm{m}}^M|\bm{x}_{\bm{q}}^{*Q},\bm{\theta}_{\bm{d}}^{*})|F_0(\bm{y}_i;\bm{\psi}_{\bm{m}}^{M(v)})-1\{\bm{y}_i\geq\lfloor\bm{y}\rfloor_{\bm{m}}^{M}\}|\nonumber\\
&\leq \left\{\max_{\bm{q}\in\mathcal{Q};\bm{d}\in\mathcal{D}^{+}}\sum_{\bm{m}\in\mathcal{M}}\sum_{k=1}^{K}H(\mathcal{Y}_{\bm{m}}^M|\bm{x}_{\bm{q}}^{*Q},\bm{\theta}_{\bm{d}}^{*})1\{y_{ik}\in\mathcal{L}_k^{\delta^{*}}(\lfloor y\rfloor_{m_k}^{M})\}\right\}+\epsilon_6\nonumber\\
&=\sum_{k=1}^{K}\left\{\max_{\bm{q}\in\mathcal{Q};\bm{d}\in\mathcal{D}^{+}}\sum_{m_k=-(M+1),\ldots,(M+1)}H_k(\mathcal{Y}_{k,m_k}|\bm{x}_{\bm{q}}^{*Q},\bm{\theta}_{\bm{d}}^{*})1\{y_{ik}\in\mathcal{L}_k^{\delta^{*}}(\lfloor y\rfloor_{m_k}^{M})\}\right\}+\epsilon_6.
\end{align}
Since $\mathcal{L}_k^{\delta^{*}}(\lfloor y\rfloor_{m_k}^{M})$ is non-overlapping for $m_k=-(M+1),\ldots,(M+1)$, only one term in the summation of Equation (\ref{apx1:step4:approx2}) is non-zero. Since $H_k$ is a continuous distribution, for any $\epsilon_7>0$, we have $H_k(\mathcal{Y}_{k,m_k}|\bm{x}_{\bm{q}}^{*Q},\bm{\theta}_{\bm{d}}^{*})\leq\epsilon_7$ for any $\bm{q}\in\mathcal{Q}$, $\bm{d}\in\mathcal{D}^{+}$ and $m_k=-(M+1),\ldots,(M+1)$ given that $M$ is sufficiently large. Finally, using the same proof idea as Equation (\ref{eq:apx1:step3:bound_final}), we have

\begin{equation} \label{eq:apx1:step4:bound_final}
|\tilde{F}^{(M,Q,t,u,v)}(\bm{y}|\bm{x})-\tilde{F}^{*(M,Q,t,u)}(\bm{y}|\bm{x})|\leq N(K\epsilon_7+\epsilon_6).
\end{equation}

In summary, based on Equations (\ref{eq:apx1:step1:bound_final}), (\ref{eq:apx1:step2:bound_final}) (\ref{eq:apx1:step3:bound_final}) and (\ref{eq:apx1:step4:bound_final}), Theorem \ref{thm:denseness_cross} holds because for sufficiently large $M$, $Q$, $t$ and $v$, the following inequality holds uniformly for each $\bm{x}_i$ falling into a compact covariates space:
\begin{equation}
|\tilde{F}^{(M,Q,t,u\rightarrow\infty,v)}(\bm{y}|\bm{x})-\tilde{H}(\bm{y}|\bm{x})|\leq
(N\epsilon_2+\epsilon_1)+\mathcal{O}_1(\epsilon_1)+N(2\epsilon_4+\epsilon_5+2\epsilon_3)+N(K\epsilon_7+\epsilon_6),
\end{equation}
where $\epsilon_1$ to $\epsilon_7$ can be chosen to be arbitrarily small, and any parameters chosen in Steps 1 to 4 are independent of $N$, $\bm{S}=(S_1,\ldots,S_L)$ and $\bm{c}(\cdot)$.

\section{Proof of Theorem \ref{thm:denseness_nested}} \label{sec:apx2}
Other than the notational differences, the proof ideas of Theorems \ref{thm:denseness_cross} and \ref{thm:denseness_nested} are substantially similar. Precisely, the 4-step framework used to prove Theorem \ref{thm:denseness_cross} also applies to prove Theorem \ref{thm:denseness_nested}. As a result, we only present a sketch proof of Theorem \ref{thm:denseness_nested}, with an emphasis on the key differences of proof techniques between the two theorems. Unless specified otherwise, the notations adopted in this proof section is the same as those defined in Appendix \ref{sec:apx1}.


Analogous to Equations (\ref{eq:apx1:1}) and (\ref{eq:apx1:2}), in Step 1 we examine an approximation bound between the following two equations:
\begin{equation} \label{eq:apx2:1}
\tilde{H}(\bm{y}|\bm{x})=\int_{\tilde{\bm{\Theta}}}\left[\prod_{{\bm{i}}\in\mathcal{I}}H(\bm{y}_{\bm{i}}|\bm{x}_{\bm{i}},\bm{\theta}_{\bm{i}})\right]dG(\tilde{\bm{\theta}}),
\end{equation}
and
\begin{equation} \label{eq:apx2:2}
\tilde{H}^{*}(\bm{y}|\bm{x})=\sum_{\tilde{\bm{d}}\in\tilde{\mathcal{D}}}\prod_{{\bm{i}}\in\mathcal{I}}H(\bm{y}_{\bm{i}}|\bm{x}_{\bm{i}},\bm{\theta}^{*}_{\bm{d}^{(\bm{i}_L)}})G(\tilde{\bm{\Theta}}_{\tilde{\bm{d}}}).
\end{equation}

Similar to Appendix \ref{sec:apx1:step1}, we choose compact spaces $\{\bar{\bm{\Theta}}_l\in\bm{\Theta}\}_{l=1,\ldots,L}$ such that $\mathbb{P}(\cap_{l=1}^{L}\cap_{\bm{i}_l\in\mathcal{I}_l}\{\bm{\theta}_{\bm{i}_l}\in\bar{\bm{\Theta}}_l\})\geq 1-NL\epsilon_1$, where $\mathcal{I}_l=\{\bm{i}_l:i_1=1,\ldots,N_0;\ldots;i_l=1,\ldots,N_{\bm{i}_{l-1}}\}$. Also partition very granular subspaces $\{\bm{\Theta}_{l,d_l}\}_{d_l=1,\ldots,D_l}$ of $\bar{\bm{\Theta}}_l$ and define the interior points $\{\bm{\theta}_{l,d_l}^*\}_{d_l=1,\ldots,D_l}$ in the exact same way as Appendix \ref{sec:apx1:step1}. Define here $\tilde{\mathcal{D}}=\prod_{l=1}^{L}\prod_{\bm{i}_l\in\mathcal{I}_l}\mathcal{D}_l^{(\bm{i}_l)}$ with $\mathcal{D}_l^{(\bm{i}_l)}=\{1,\ldots,D_l\}$, $\tilde{\bm{d}}=\{d_l^{(\bm{i}_l)}\}_{l=1,\ldots,L;\bm{i}_l\in\mathcal{I}_l}$ with $d_l^{(\bm{i}_l)}\in\mathcal{D}_l^{(\bm{i}_l)}$, $\bm{d}^{(\bm{i}_L)}=\{d_l^{(\bm{i}_l)}\}_{l=1,\ldots,L}$, and $\tilde{\bm{\Theta}}_{\tilde{\bm{d}}}=\prod_{l=1}^{L}\prod_{\bm{i}_l\in\mathcal{I}_l}\bm{\Theta}_{l,d_l^{(\bm{i}_l)}}$. Then, using the exact logic as Equations (\ref{eq:apx1:step1:rewrite1}) to (\ref{eq:apx1:step1:bound_final}), an arbitrarily small approximation error bound between Equations (\ref{eq:apx2:1}) and (\ref{eq:apx2:2}) can be obtained.


In Step 2, we define the following function analogous to Equation (\ref{eq:apx1:3}) and derive its error bound for approximating Equation (\ref{eq:apx2:2}):

\begin{equation} \label{eq:apx2:3}
\tilde{F}^{**(u)}(\bm{y}|\bm{x})=\int\prod_{{\bm{i}}\in\mathcal{I}}F^{**(u)}(\bm{y}_{\bm{i}}|\bm{x}_{\bm{i}},\bm{w}_{\bm{i}})d\Phi (\bm{w})
\end{equation}
with $F^{**(u)}(\bm{y}_{\bm{i}}|\bm{x}_{\bm{i}},\bm{w}_{\bm{i}})=\sum_{\bm{d}\in\mathcal{D}^{+}}\xi_{\bm{d}}^{(u)}(\bm{w}_{\bm{i}})H(\bm{y}_{\bm{i}}|\bm{x}_{\bm{i}},\bm{\theta}^{*}_{\bm{d}})$. Here, we choose $\xi_{\bm{d}}^{(u)}(\bm{w}_{\bm{i}})$ in a different way as Equation (\ref{eq:apx1:step2:xi}), which is crucial to cater for the dependencies of random effects across levels, as follows:

\begin{equation} \label{eq:apx2:step2:xi}
\xi_{\bm{d}}^{(u)}(\bm{w}_{\bm{i}})=\exp\{\sum_{l=1}^{L}u^{1/l}(\tilde{\beta}_{\bm{d}_l,0}+\tilde{\beta}_{\bm{d}_l,1}w_{il})\}/\sum_{\bm{d}'\in\mathcal{D}^{+}}\exp\{\sum_{l=1}^{L}u^{1/l}(\tilde{\beta}_{\bm{d}'_l,0}+\tilde{\beta}_{\bm{d}'_l,1}w_{il})\},
\end{equation}
where $\mathcal{D}=\prod_{l=1}^{L}\{1,\ldots,D_l\}$ and $\mathcal{D}^{+}=\prod_{l=1}^{L}\{0,1,\ldots,D_l+1\}$ are exactly the same as those defined in Appendix \ref{sec:apx1:step2}, and $\bm{d}_l=(d_1,\ldots,d_l)$ and $\bm{d}'_l=(d'_1,\ldots,d'_l)$ with $\bm{d}=\bm{d}_L$ (and $\bm{d}'=\bm{d}'_L$).


In contrast to Appendix \ref{sec:apx1:step2}, we construct intervals $\mathcal{W}_{l,\bm{d}_l}$ such that $\cup_{d_l=0}^{D_l+1}\mathcal{W}_{l,\bm{d}_l}=\mathbb{R}$ for any $\bm{d}_{l-1}\in\prod_{l'=1}^{l-1}\{1,\ldots,D_{l'}\}$ and $\Phi_l(\mathcal{W}_{l,\bm{d}_l})=G_l(\bm{\Theta}_{l,d_l}|\bm{\Theta}_{1,d_1},\ldots,\bm{\Theta}_{l-1,d_{l-1}})$, which represents the probability of level-$l$ random effect belongs to $\bm{\Theta}_{l,d_l}$ conditioned on the corresponding upper level random effects fall into $(\bm{\Theta}_{1,d_1},\ldots,\bm{\Theta}_{l-1,d_{l-1}})$. The following lemma is analogous to Lemma \ref{lm:xi}:

\begin{lemma} \label{lm:xi2}
There exists parameters $\{(\tilde{\beta}_{\bm{d}_l,0},\tilde{\beta}_{\bm{d}_l,1})\}_{{\bm{d}}\in\mathcal{D}^{+};l=1,\ldots,L}$ of $\xi^{(u)}_{\bm{d}}(\bm{w}_{\bm{i}})$ such that $\xi^{(u)}_{\bm{d}}(\bm{w}_{\bm{i}})\xrightarrow{u\rightarrow\infty}\prod_{l=1}^{L}1^{*}_{w_{il}}(\mathcal{W}_{l,\bm{d}_l})$ for every $\bm{d}\in\mathcal{D}^{+}$.
\end{lemma}

\begin{proof}
Similar to the proof of Lemma \ref{lm:xi}, we choose suitable parameters such that $\tilde{\beta}_{\bm{d}_l,0}+\tilde{\beta}_{\bm{d}_l,1}w_{il}>\max_{d_l'\neq d_l}\tilde{\beta}_{\bm{d}_l^{*},0}+\tilde{\beta}_{\bm{d}_l^{*},1}w_{il}$ if and only if $w_{il}\in\mathcal{W}_{l,\bm{d}_l}^{*}$ for every $\bm{d}\in\mathcal{D}^{+}$ and $l=1,\ldots,L$, where $\mathcal{W}_{l,\bm{d}_l}^{*}$ is the interior of $\mathcal{W}_{l,\bm{d}_l}$ and $\bm{d}_l^{*}=(d_1,\ldots,d_{l-1},d_l')$. Observe the expression $\sum_{l=1}^{L}u^{1/l}(\tilde{\beta}_{\bm{d}_l,0}+\tilde{\beta}_{\bm{d}_l,1}w_{il})$ in Equation (\ref{eq:apx2:step2:xi}) where the term corresponding to a higher level factor (small $l$) dominates when $u$ is large. Therefore, for sufficient large $u$, we have $\sum_{l=1}^{L}u^{1/l}(\tilde{\beta}_{\bm{d}_l,0}+\tilde{\beta}_{\bm{d}_l,1}w_{il})>\max_{\bm{d}\neq\bm{d}'}\sum_{l=1}^{L}u^{1/l}(\tilde{\beta}_{\bm{d}'_l,0}+\tilde{\beta}_{\bm{d}'_l,1}w_{il})$ for $\bm{w}_i$ satisfying $w_{il}\in\mathcal{W}_{l,\bm{d}_l}^{*}$ for every $l=1,\ldots,L$. Referring to the same logic as the proof of Lemma \ref{lm:xi} and the result follows.
\end{proof}

Now, the approximation bound between Equations (\ref{eq:apx2:2}) and (\ref{eq:apx2:3}) can be obtained using the same logic (with notational variations) as that outlined by Equation (\ref{eq:apx1:step2:bound_final}), where we further note that $\Phi (\mathcal{W}_{\tilde{\bm{d}}})=\prod_{l=1}^L\prod_{\bm{i}_l\in\mathcal{I}_l}\Phi_l(\mathcal{W}_{l,\bm{d}_l^{(\bm{i}_l)}})=\prod_{l=1}^L\prod_{\bm{i}_l\in\mathcal{I}_l}G_l(\bm{\Theta}_{l,d_l^{(\bm{i}_{l})}}|\bm{\Theta}_{1,d_1^{(\bm{i}_{1})}},\ldots,\bm{\Theta}_{l-1,d_{l-1}^{(\bm{i}_{l-1})}})=G(\tilde{\bm{\Theta}}_{\tilde{\bm{d}}})$ with $\mathcal{W}_{\tilde{\bm{d}}}=\prod_{l=1}^L\prod_{\bm{i}_l\in\mathcal{I}_l}\mathcal{W}_{l,\bm{d}_l^{(\bm{i}_l)}}$.


Step 3 and 4 involve evaluations of the following expressions in analogous to Equations (\ref{eq:apx1:4}) and (\ref{eq:apx1:5}):

\begin{equation} \label{eq:apx2:4}
\tilde{F}^{*(M,Q,t,u)}(\bm{y}|\bm{x})=\int\prod_{\bm{i}\in\mathcal{I}}F^{*(M,Q,t,u)}(\bm{y}_{\bm{i}}|\bm{x}_{\bm{i}},\bm{w}_{\bm{i}})d\Phi (\bm{w})
\end{equation}
with $F^{*(M,Q,t,u)}(\bm{y}_{\bm{i}}|\bm{x}_{\bm{i}},\bm{w}_{\bm{i}})=\sum_{\bm{m}\in\mathcal{M}}\sum_{\bm{q}\in\mathcal{Q}}\sum_{\bm{d}\in\mathcal{D}^{+}}\pi_j^{(t,u)}(\bm{x}_{\bm{i}},\bm{w}_{\bm{i}};\tilde{\bm{\alpha}}_{\bm{q}}^Q,\tilde{\bm{\beta}}_{\bm{d}})1\{\bm{y}_{\bm{i}}\geq\lfloor\bm{y}\rfloor_{\bm{m}}^{M}\}$,
\begin{equation} \label{eq:apx2:5}
\tilde{F}^{(M,Q,t,u,v)}(\bm{y}|\bm{x})=\int\prod_{\bm{i}\in\mathcal{I}}F^{(M,Q,t,u,v)}(\bm{y}_{\bm{i}}|\bm{x}_{\bm{i}},\bm{w}_{\bm{i}})d\Phi (\bm{w})
\end{equation}
with $F^{(M,Q,t,u,v)}(\bm{y}_{\bm{i}}|\bm{x}_{\bm{i}},\bm{w}_{\bm{i}})=\sum_{\bm{m}\in\mathcal{M}}\sum_{\bm{q}\in\mathcal{Q}}\sum_{\bm{d}\in\mathcal{D}^{+}}\pi_j^{(t,u)}(\bm{x}_{\bm{i}},\bm{w}_{\bm{i}};\tilde{\bm{\alpha}}_{\bm{q}}^Q,\tilde{\bm{\beta}}_{\bm{d}})F_0(\bm{y}_{\bm{i}};\bm{\psi}_{\bm{m}}^{M(v)})$. The derivation techniques here are exactly the same as those presented in Appendices \ref{sec:apx1:step3} and \ref{sec:apx1:step4}, so this part of the proof is omitted.

\end{appendices}

\end{document}